\RequirePackage{fix-cm}
\documentclass[jap,press]{apt}
\firstpage{1}
\numberwithin{equation}{section}
\usepackage{mathtools}
\usepackage{enumitem}
\usepackage{array}
\usepackage{needspace}
\usepackage[numbers]{natbib}
\usepackage{etoolbox}
\usepackage[colorlinks=true,linkcolor=blue,urlcolor=blue,citecolor=blue]{hyperref}

\renewcommand{\crnotice}{}
\paperno{}
\shorttitle{Minorization and ergodicity with countable switching}
\patchcmd{\ams}{2010}{2020}{}{\errmessage{MSC year patch failed}}
\patchcmd{\ams}{2010}{2020}{}{\errmessage{MSC year patch failed}}
\setlist{itemsep=2pt,topsep=4pt,parsep=0pt,partopsep=0pt}
\hypersetup{
 pdftitle={Minorization and Exponential Ergodicity for Degenerate Diffusions with Countable Switching},
 pdfsubject={Ergodicity of degenerate regime-switching diffusions},
 pdfkeywords={regime-switching diffusion, degenerate diffusion, local minorization,
 controllability, exponential ergodicity, state-dependent switching}
}

\newenvironment{remark*}[1][]{\par\addvspace{\medskipamount}\noindent
 \textbf{Remark.}\ifstrempty{#1}{}{\textit{ (#1.)}}\ \ignorespaces}
 {\par\medbreak}
\newenvironment{proposition*}[1][]{\par\addvspace{\medskipamount}\noindent
 \textbf{Proposition.}\ifstrempty{#1}{}{\textup{ (#1.)}}\ \itshape\ignorespaces}
 {\par\normalfont\medbreak}

\newcommand{\R}{\mathbb R}

\newcommand{\Nzero}{\mathbb N_0}
\newcommand{\Sset}{\mathcal S}
\newcommand{\one}{\mathbf 1}
\newcommand{\Leb}{\operatorname{Leb}}

\newcommand{\B}{\mathcal B}
\newcommand{\A}{\mathcal A}

\newcommand{\Ran}{\operatorname{Ran}}
\newcommand{\PP}{\mathbb P}
\newcommand{\E}{\mathbb E}

\begin{document}
\raggedbottom
\title{Minorization and exponential ergodicity\\
for degenerate diffusions with countable switching}
\par
\authornames{YUZHONG CHENG}
\authoronetwo[Kyushu University]{Yuzhong Cheng}
\addressone{Institute of Mathematics for Industry, Kyushu University,
744 Motooka, Nishi-ku, Fukuoka 819-0395, Japan.}
\emailtwo{cheng.yuzhong.451@m.kyushu-u.ac.jp}
\hypersetup{pdfauthor={Yuzhong Cheng}}

\begin{abstract}
We establish a local-minorization criterion for diffusions with countably
many state-dependent switching regimes, even when every frozen transition
law is singular. One regular controlled path yields a common absolutely
continuous component for nearby initial states, with rate positivity
required only at the switching locations. The proof controls a family of
Brownian perturbations and its derivatives while retaining a path
restriction for the switching likelihood. Controlled entrance and a joint
Lyapunov inequality then give weighted exponential ergodicity. A concrete
coefficient and rank criterion and degenerate network examples illustrate
the result.
\end{abstract}
\keywords{Regime-switching diffusion; degenerate diffusion; local minorization;
controllability; exponential ergodicity; state-dependent switching}
\ams{60J60}{60J25; 60J27}

\section{Introduction}\label{sec:introduction}
Regime-switching diffusions combine Brownian motion with a discrete
environment whose transition rates depend on the continuous state.
We consider
\[
dX_t=b_{\Lambda_{t-}}(X_t)\,dt+
\sigma_{\Lambda_{t-}}(X_t)\,dW_t,\qquad X_t\in\R^d,
\]
where $W$ is $m$-dimensional Brownian motion, $\Lambda$ takes values
in a countably infinite set $\Sset$, and its transition rate from
$i$ to $j$ at spatial position $x$ is $q_{ij}(x)$.
The precise construction is given in Section~\ref{sec:model}.
Our question is how the joint process can be exponentially ergodic
in weighted total variation when every frozen diffusion has singular
transition laws. The main result produces a common absolutely
continuous component from one regular controlled path. It requires
positivity of the selected switching rates only at the switching
locations.

Frozen-regime regularization is a common route to total-variation
ergodicity. For countably many state-dependent regimes,
\citet[Assumption~4.1 and Theorem~6.3]{XiZhu2017} use frozen strong
Feller properties and transition densities, irreducibility, and a
Foster--Lyapunov condition; \citet[Assumption~3.2 and Theorem~3.2]{Kunwai2022}
uses ellipticity and Lyapunov criteria.
For finitely many regimes and unbounded rates,
\citet[Theorem~3.4]{TongMajda2016} combine averaged dissipation with
minorization in an accessible frozen regime.
\citet{LazicSandric2022} obtain subgeometric total-variation results
with irreducibility verified by local ellipticity and accessibility.
For smooth densities, \citet[(14) and Theorem~4.6]{HuNualartSunXie2019}
assume a uniform H\"ormander condition in every regime.

Other approaches use contraction or impose minorization directly
on the joint process. \citet[Theorem~3.8]{BasakBisiGhosh1999} use
asymptotic flatness of an embedded flow to obtain uniqueness and weak
convergence without ellipticity or frozen densities.
\citet{Shao2015Wasserstein} uses coupling and $M$-matrix criteria
for Wasserstein exponential ergodicity, including countably many
regimes. \citet{CloezHairer2015} obtain Wasserstein contraction from
contraction on average, and total-variation estimates with a frozen
small-set condition. For finite-state independent switching,
\citet[Theorems~3.5 and~4.5]{LiXi2022TV} combine scalar drift conditions
with minorization of the joint process.
Switching itself also produces regularization in deterministic-flow
models: \citet{BakhtinHurth2012} prove uniqueness and absolute
continuity of invariant measures, and \citet{BenaimEtAl2015} obtain
local minorization and exponential total-variation convergence under
accessibility and bracket conditions.

Singular covariance alone does not preclude frozen regularization.
The switching Hamiltonian systems of \citet{XiZhuWu2021}, forced in
the velocity variables, and the switching SIR model of
\citet{JinLinPei2018}, driven by one Brownian motion, both use frozen
regularization in their ergodic arguments. Our network example in
Section~\ref{sec:examples} has the stronger obstruction that every
frozen transition law is singular. Its continuous component is
\[
dX_t=(A_{E_{t-}}X_t+c_{E_{t-}})\,dt
+[g_{E_{t-}}(X_t)+s_{N_{t-}}]\mathbf v_1\circ dW_t,
\qquad X_t\in\R^N,\quad N\ge3.
\]
Here $E_t$ selects an edge, $N_t$ is a countably valued noise level,
and $\mathbf v_1$ is the first coordinate vector. The matrices $A_e$
combine coupling along the active edge, damping at node~1, and a
common linear growth term. The model uses the pairwise interactions
of \citet{BoydEtAl2006} and switching pinning of \citet{HanEtAl2014},
with Brownian forcing at node~1. An unforced projection evolves
deterministically in every fixed regime, so each frozen transition
law is supported on an affine hyperplane; see \eqref{eq:frozen-projection}.
Switching must spread both noise and damping through the network.

Our strategy is to establish local minorization for such degenerate
systems by transferring regular controlled endpoints to the stochastic
process. Controlled entrance and a joint Lyapunov inequality then
give exponential ergodicity through Harris' theorem.

The main contributions are as follows. \textup{(1)}
Theorem~\ref{prop:local} gives local minorization from one admissible
control with a surjective endpoint differential on a prescribed
switching schedule. \textup{(2)} The stochastic argument preserves
an entire $C^1$ family of controlled paths under multiplicative
Brownian forcing, uniformly for nearby initial points and switching
times. This controls both the endpoint Jacobian and the
state-dependent switching likelihood within the same path
neighborhood. \textup{(3)} A coefficient representation, a rank
condition across regimes, and positive-rate graph accessibility give
a concrete sufficient criterion.
Under nonexplosion, they yield one minorization time and target
measure for every compact initial set; only the lower-bound constant
depends on that set (Theorem~\ref{thm:common-time-minorization}).

The coefficient criterion assumes, on a finite set $\mathcal J$,
\[
b_i^\circ(x)=A_ix+c_i+\sigma_i(x)f_i(x),\qquad
\sigma_i(x)R_i(x)=B_i,
\]
with constant $A_i,B_i,c_i$, $C^2$ maps $f_i,R_i$, and Stratonovich
drift $b_i^\circ$. The control $u=R_i(\phi)v-f_i(\phi)$ reduces the
controlled equation to $\dot\phi=A_i\phi+c_i+B_iv$.
A rank condition combining different regimes then applies even when
no individual linear system is controllable.
Proposition~\ref{prop:regular-word} adapts the switched-linear
arguments of \citet[Section~4.2, (25)--(33)]{SunGeLee2002},
\citet{LiuLinChen2013}, and \citet[Theorem~3.1]{Sun2007Constrained}
to respect the positive-rate graph and lift the controls to the
original equation. This is a structured sufficient criterion in the
direction of the conjectural cross-regime bracket extension discussed by
\citet[Section~5, discussion item~3]{TongMajda2016}.

The control-to-minorization principle is established in earlier work.
\citet{BenArousLeandre1991} and
\citet[Lemma~4.5 and Appendix~A]{HerzogMattingly2015} relate diffusion
density positivity to regular controlled endpoints.
\citet[Theorem~1.1 and Proposition~3.2]{Shirikyan2017} uses
controllability and finite-dimensional changes of variables for
total-variation mixing of discrete-time random systems.
\citet{NersesyanRaquepas2021} obtain exponential total-variation
mixing for degenerate Poisson-driven SDEs from dissipativity and
approximate and solid controllability; their Lemma~2.3 gives a common
component near a regular endpoint and a vector of jump times.
\citet{MertzNersesyanRissel2024} develop a controllability criterion
for exponential mixing of constrained random dynamical systems,
including white-noise applications. Their framework assumes a
continuous resolving map with local differentiability in the noise.

For the present model, applying this approach requires a lower bound
that remains valid after averaging over the state-dependent switching
law. Lemma~\ref{lem:classical-family} combines the flow regularity of
\citet{Kunita1986Flows} and the support theorem of
\citet{Aida1990FlowSupport} to control paths and their parameter
derivatives simultaneously, with continuity in the initial point and
switching times. Trajectory support alone does not supply this
derivative control. Lemma~\ref{lem:averaging} turns these estimates
into a common Lebesgue lower bound while retaining the path
restriction needed for the switching likelihood.
Only finitely many regimes enter this construction. A proper joint
Lyapunov function handles the countable regime tail, without uniform
coefficient bounds over $\Sset$; the network has unbounded
regime-dependent noise amplitudes. The full transition law can still
retain a singular component.

The paper is organized as follows. Section~\ref{sec:model} states
the local-minorization theorem, its ergodic consequence, and the
coefficient criterion. Section~\ref{sec:rank-cycle} proves the
control construction, local minorization, and controlled entrance.
Section~\ref{sec:examples} gives examples separating the criteria
from global rate positivity and frozen-regime regularization.
Appendix~\ref{app:analytic} contains the analytic lemmas.

\section{Assumptions and main results}\label{sec:model}
\subsection{Model and local assumptions}
Identify $\Sset$ with $\Nzero$ and give it the discrete topology.
Set $E=\R^d\times\Sset$.
For each $i\in\Sset$, the drift $b_i:\R^d\to\R^d$ and diffusion
coefficient $\sigma_i:\R^d\to\R^{d\times m}$ are Borel maps.
The switching rates $q_{ij}:\R^d\to[0,\infty)$, $j\ne i$, are Borel and satisfy
\[
q_i(x):=\sum_{j\ne i}q_{ij}(x)<\infty,\qquad q_{ii}(x):=-q_i(x).
\]
For each $x$, $Q(x)=(q_{ij}(x))_{i,j\in\Sset}$ is a conservative
rate matrix. For fixed $i$, list $\Sset\setminus\{i\}$ in increasing
order as $j_1,j_2,\ldots$. Define consecutive left-closed, right-open
intervals of lengths $q_{i j_r}(x)$ and the corresponding jump function by
\[
\begin{aligned}
\Gamma_{i j_r}(x)
&=\left[\sum_{\ell=1}^{r-1}q_{i j_\ell}(x),
        \sum_{\ell=1}^{r}q_{i j_\ell}(x)\right),\qquad r\ge1,\\
h(x,i,u)&=\sum_{j\ne i}(j-i)\one_{\Gamma_{ij}(x)}(u),\qquad u\ge0.
\end{aligned}
\]
This is the Poisson representation used in \citet[Section~2]{XiZhu2017}.
The Borel function $h$ vanishes for $u\notin[0,q_i(x))$.

On a complete filtered probability space
$(\Omega,\mathcal F,(\mathcal F_t)_{t\ge0},\PP)$ satisfying the usual
conditions, let $W$ be an $m$-dimensional $(\mathcal F_t)$-Brownian
motion and let $\mathsf N(dt,du)$ be an independent $(\mathcal F_t)$-Poisson
random measure on $[0,\infty)^2$ with compensator $dt\,du$.
For each initial state $(x,i)\in E$, we study
\begin{equation}\label{eq:model-general}
\begin{aligned}
dX_t&=b_{\Lambda_{t-}}(X_t)\,dt+
       \sigma_{\Lambda_{t-}}(X_t)\,dW_t,\qquad X_0=x,\\
d\Lambda_t&=\int_0^\infty h(X_t,\Lambda_{t-},u)\,\mathsf N(dt,du),
       \qquad \Lambda_0=i.
\end{aligned}
\end{equation}
Let $\zeta$ denote the maximal lifetime of the solution.
For distinct $i,j$, the transition count and its compensator on $[0,\zeta)$ are
\[
N_{ij}(t):=\sum_{0<s\le t}\one_{\{\Lambda_{s-}=i,\,\Lambda_s=j\}},
\qquad
\int_0^t\one_{\{\Lambda_{s-}=i\}}q_{ij}(X_s)\,ds,
\]
respectively.
For $g:\Sset\to\R$ and $\varphi\in C^2(\R^d)$, write
\[
\begin{aligned}
Qg(x,i)&=\sum_{j\ne i}q_{ij}(x)(g(j)-g(i)),\\
\mathcal L_i\varphi(x)&=b_i(x)\cdot\nabla\varphi(x)
+\tfrac12\operatorname{tr}\bigl(a_i(x)D^2\varphi(x)\bigr),\qquad
 a_i(x)=\sigma_i(x)\sigma_i(x)^\top,
\end{aligned}
\]
where the sum defining $Qg$ is required to converge absolutely.
The operator $\mathcal L_i$ describes the diffusion with regime $i$ held fixed.

For coefficient maps $F:\R^d\to\R^q$, $DF(x)$ denotes the
$q\times d$ Jacobian with respect to $x$.
Write $\sigma_i^{(k)}$ for the $k$th column of $\sigma_i$ and set
\begin{equation}\label{eq:corrected}
b_i^\circ=b_i-\frac12\sum_{k=1}^mD\sigma_i^{(k)}\,\sigma_i^{(k)}.
\end{equation}
\Needspace{6\baselineskip}
\begin{assumpt}[Regularity and switching rates]\label{ass:local-construction}
\leavevmode
\begin{enumerate}[label=\textup{(\roman*)},leftmargin=*]
\item For every $i\in\Sset$,
$b_i\in C^{d+3}(\R^d;\R^d)$ and
$\sigma_i\in C^{d+4}(\R^d;\R^{d\times m})$.
\item Each $q_{ij}$, $i\ne j$, is continuous and nonnegative, and
\[
\sup_{x\in K,\,i\in F}q_i(x)<\infty
\quad\text{for every compact }K\subset\R^d
\text{ and finite }F\subset\Sset.
\]
\end{enumerate}
\end{assumpt}
The differentiability orders in item~(i) ensure the $C^1$ support
property for the augmented $2d$-dimensional flow in
Lemma~\ref{lem:classical-family}.

The coefficients in item~(i) are locally Lipschitz. The interlacing
construction of \citet[proof of Theorem~2.1, Step~1, (2.2)--(2.10)]{XiZhu2017},
localized to compact subsets of $E$, therefore gives a pathwise unique
maximal strong Markov solution $Z_t=(X_t,\Lambda_t)$ with lifetime
$\zeta$. Its lifetime cannot end while it remains in a compact set.
Write $\PP_z,\E_z$ for its law and expectation from $z$, and
$P_t(z,A)=\PP_z(t<\zeta,\,Z_t\in A)$.

For $f:E\to\R$ with $f(\cdot,i)\in C^2(\R^d)$ for every $i$, define
\begin{align}
J_f(x,i)&=\sum_{j\ne i}q_{ij}(x)|f(x,j)-f(x,i)|,
\label{eq:local-jump-sum}\\
\A f(x,i)&=\mathcal L_i f(\cdot,i)(x)
+\sum_{j\ne i}q_{ij}(x)(f(x,j)-f(x,i)),
\label{eq:joint-generator}
\end{align}
whenever $J_f(x,i)<\infty$. A nonnegative function is called proper
when all its finite sublevel sets are compact in $E$.

\subsection{Controlled paths and local minorization}\label{subsec:controlled-paths}
For $T>0$, a \emph{prescribed switching schedule} on $[0,T]$ is
a deterministic piecewise constant function
$\lambda=(\lambda_t)_{0\le t\le T}$ with values in $\Sset$, specified
by a finite regime sequence
$\boldsymbol\iota=(\iota_0,\ldots,\iota_n)\in\Sset^{n+1}$,
$n\in\Nzero$, with $\iota_{r-1}\ne\iota_r$ for $1\le r\le n$, and switching times
\[
\tau\in\Delta_n(T):=\{\tau\in\R^n:0<\tau_1<\cdots<\tau_n<T\},
\qquad \tau_0=0,\quad\tau_{n+1}=T,
\]
through
\[
\lambda_t=\lambda_t^{\boldsymbol\iota,\tau}=\iota_r
\quad\text{for }t\in[\tau_r,\tau_{r+1}),\quad 0\le r\le n,
\qquad \lambda_T=\iota_n.
\]
Here \emph{prescribed} means that $n,\boldsymbol\iota,\tau$ are
deterministic choices made before solving the controlled equation.
These data are chosen separately for each prescribed switching schedule.
Write $\Leb_n$ for Lebesgue measure on $\R^n$, with
$\Delta_0(T)=\{()\}$ and $\Leb_0(\{()\})=1$.

Given $x\in\R^d$, a prescribed switching schedule $\lambda=\lambda^{\boldsymbol\iota,\tau}$
on $[0,T]$, and a control $u\in L^2([0,T];\R^m)$, let
$\phi^{x,\lambda,u}$ be the maximal absolutely continuous solution of
the controlled equation below. We write $\phi=\phi^{x,\lambda,u}$
when these data are fixed. The endpoint map is defined whenever
this solution exists throughout $[0,T]$:
\begin{equation}\label{eq:skeleton}
\begin{aligned}
\dot\phi_t&=b_{\lambda_t}^\circ(\phi_t)
 +\sigma_{\lambda_t}(\phi_t)u_t,\qquad \phi_0=x,\\
\mathcal E_{x,T,\lambda}(u)&=\phi_T.
\end{aligned}
\end{equation}
For the prescribed switching schedule $\lambda$, call $u$ \emph{admissible from
$x$ on $[0,T]$} if $\phi^{x,\lambda,u}$ exists on $[0,T]$ and
\begin{equation}\label{eq:admissible-controlled-path}
q_{\iota_{r-1}\iota_r}(\phi^{x,\lambda,u}_{\tau_r})>0,
\qquad 1\le r\le n.
\end{equation}
The resulting joint controlled path is
$t\mapsto(\phi^{x,\lambda,u}_t,\lambda_t)$, from
$(x,\iota_0)$ to $(\phi^{x,\lambda,u}_T,\iota_n)$.
For $n=0$, the rate condition is empty.
An admissible control $u$ is \emph{regular} for this initial point
and prescribed switching schedule if
\begin{equation}\label{eq:regular-controlled-path}
\Ran D_u\mathcal E_{x,T,\lambda}(u)=\R^d.
\end{equation}

The differential in \eqref{eq:regular-controlled-path} varies the
Brownian control $u$ while holding $x,T,\lambda$ fixed. In particular,
if $\sigma_i\equiv0$ in every visited regime, this differential is
zero. 
Write $\Leb$ for Lebesgue measure on $\R^d$ and $\B(\R^d)$ for
its Borel $\sigma$-field.

\begin{thm}[Fixed-time local minorization]\label{prop:local}
Suppose Assumption~\ref{ass:local-construction} holds.
Fix $p\in\R^d$, $T_0>0$, and a prescribed switching schedule
$\lambda=\lambda^{\boldsymbol\iota,\tau^*}$ on $[0,T_0]$.
Let $u_*\in L^2([0,T_0];\R^m)$ be admissible and regular from
$p$ under this prescribed switching schedule, and set
$\phi=\phi^{p,\lambda,u_*}$ and $y_*=\phi_{T_0}$.
There are a neighborhood $U$ of $p$, a ball $B_*$ about $y_*$, and
$c_*>0$ such that
\begin{equation}\label{eq:local}
P_{T_0}((x,\iota_0),D\times\{\iota_n\})
\ge c_*\Leb(D\cap B_*),
\qquad x\in U,\quad D\in\B(\R^d).
\end{equation}
\end{thm}

Theorem~\ref{prop:local} is the main local result. Its rate condition
\eqref{eq:admissible-controlled-path} is imposed only at the selected
switching locations. 

\subsection{Exponential ergodicity from controlled entrance}
\begin{assumpt}\label{ass:joint-lyapunov}
There are a proper function $V:E\to[1,\infty)$ with
$V(\cdot,i)\in C^2(\R^d)$ for every $i\in\Sset$, and constants
$a>0$, $b\ge0$, such that $J_V$ is locally bounded and
\begin{equation}\label{eq:joint-lyapunov-drift}
\A V\le-aV+b\quad\text{on }E.
\end{equation}
\end{assumpt}

For a signed measure $\mu$ with $\int_E V\,d|\mu|<\infty$,
write $\|\mu\|_V=\sup_{|f|\le V}|\int_Ef\,d\mu|$, with the
supremum over measurable real functions.

Write $C_R=\{z\in E:V(z)\le R\}$ for the Lyapunov function in
Assumption~\ref{ass:joint-lyapunov}. 

\begin{thm}[Ergodicity from controlled paths]\label{thm:controlled-ergodicity}
Suppose Assumptions~\ref{ass:local-construction} and
\ref{ass:joint-lyapunov} hold. Assume there are $(p_*,\ell_*)\in E$
and $R>2b/a$, with $C_R\ne\varnothing$, such that:
\begin{enumerate}[label=\textup{(\roman*)},leftmargin=*]
\item There exist $T_0>0$, a prescribed switching schedule $\lambda^*$ on $[0,T_0]$
with $\lambda^*_0=\ell_*$, and a control
$u_*\in L^2([0,T_0];\R^m)$ that is admissible and regular from
$p_*$ under $\lambda^*$.
\item For every open neighborhood $U$ of $p_*$, there exists
$T_{\rm a}>0$ such that, for each $z=(x,i)\in C_R$, one can choose
a prescribed switching schedule $\lambda^z$ on $[0,T_{\rm a}]$ and a control
$u^z\in L^2([0,T_{\rm a}];\R^m)$ admissible from $x$ under
$\lambda^z$, satisfying
\[
\lambda^z_0=i,\qquad \lambda^z_{T_{\rm a}}=\ell_*,\qquad
\phi^{x,\lambda^z,u^z}_{T_{\rm a}}\in U.
\]
\end{enumerate}
Then the joint process is nonexplosive and has a unique invariant
probability $\pi$, with $\pi(V)\le b/a$. There are constants
$M<\infty$ and $\rho>0$ such that
\begin{equation}\label{eq:weighted-ergodicity}
\|P_t(z,\cdot)-\pi\|_V\le MV(z)e^{-\rho t},
\qquad t\ge0,\quad z\in E.
\end{equation}
\end{thm}

\begin{proof}[Proof of Theorem~\ref{thm:controlled-ergodicity}]
Fix $z\in E$ and let
$\tau_N=\inf\{t<\zeta:V(Z_t)\ge N\}$ for integers $N>V(z)$,
with $\inf\varnothing=\infty$.
Properness of $V$ and the local construction give $\tau_N\uparrow\zeta$
and $\tau_N<\zeta$ on $\{\zeta<\infty\}$.
The stopped Brownian integrand is bounded, and the compensation formula gives
\[
\E_z\sum_{0<s\le t\wedge\tau_N}|V(Z_s)-V(Z_{s-})|
 =\E_z\int_0^{t\wedge\tau_N}J_V(Z_{s-})\,ds
 \le t\sup_{\{V\le N\}}J_V<\infty.
\]
The compensation estimate includes the possibly unbounded exit
jump, so the stopped martingale terms have mean zero. It\^o's
formula and \eqref{eq:joint-lyapunov-drift} give
\[
\begin{aligned}
\E_z[e^{a(t\wedge\tau_N)}V(Z_{t\wedge\tau_N})]
 &\le V(z)+\frac ba(e^{at}-1),\\
\PP_z(\zeta\le t)\le\PP_z(\tau_N\le t)
 &\le\frac{V(z)+(b/a)(e^{at}-1)}{N}\longrightarrow0.
\end{aligned}
\]
Hence the process is nonexplosive. Fatou's lemma gives
\begin{equation}\label{eq:joint-semigroup-drift}
P_tV\le e^{-at}V+\frac ba(1-e^{-at}),\qquad t\ge0.
\end{equation}

Apply Theorem~\ref{prop:local} with $p=p_*$ and the prescribed switching schedule
$\lambda^*$ and control $u_*$ from item~(i). It gives an open
neighborhood $U$ of $p_*$, a probability
measure $\nu$ supported on a ball in the terminal regime, and
$\varepsilon_0>0$ such that
\[
P_{T_0}((x,\ell_*),\cdot)\ge\varepsilon_0\nu(\cdot),
\qquad x\in U.
\]
Choose $T_{\rm a}$ from item~(ii) for this $U$.
Lemma~\ref{lem:controlled-entrance} gives
$\alpha_R:=\inf_{z\in C_R}P_{T_{\rm a}}(z,U\times\{\ell_*\})>0$.
The Markov property, with $h=T_{\rm a}+T_0$, yields
\[
P_h(z,\cdot)\ge\alpha_R\varepsilon_0\nu(\cdot),\qquad z\in C_R.
\]
Put $\mathsf P=P_h$. By \eqref{eq:joint-semigroup-drift},
\[
\mathsf PV\le\gamma V+K_h,\qquad
\gamma=e^{-ah}\in(0,1),\qquad K_h=\frac ba(1-\gamma),\qquad
\frac{2K_h}{1-\gamma}=\frac{2b}{a}<R.
\]
Thus the sublevel set chosen in the theorem satisfies the Harris
threshold for this $h$. Reduce $\alpha_R\varepsilon_0$ if necessary
to obtain a minorization constant in $(0,1)$.
The Harris theorem of
\citet[Theorem~1.2 and Section~3.1]{HairerMattingly2011Harris}
therefore gives a unique invariant probability $\pi$ for $\mathsf P$,
with $\pi(V)<\infty$. Since $V\le1+V\le2V$, its weighted convergence
and the integrated drift bound give
\[
\|\mathsf P^n(z,\cdot)-\pi\|_V\le M_h\theta^nV(z),
\qquad
\pi(V)\le\frac{K_h}{1-\gamma}=\frac ba,
\]
for every $n\ge0$, with $M_h<\infty$ and $\theta\in(0,1)$.

For every $s\ge0$, $(\pi P_s)\mathsf P=(\pi\mathsf P)P_s=\pi P_s$,
so uniqueness gives $\pi P_s=\pi$; any invariant probability for the
full semigroup is likewise invariant for $\mathsf P$ and equals $\pi$.
For $t=nh+s$, $0\le s<h$, \eqref{eq:joint-semigroup-drift} gives
$P_sV\le(1+b/a)V$. With $\rho=-h^{-1}\log\theta>0$,
\[
\begin{aligned}
\|P_t(z,\cdot)-\pi\|_V
&=\|\bigl(\mathsf P^n(z,\cdot)-\pi\bigr)P_s\|_V\\
&\le(1+b/a)M_h\theta^nV(z)
 \le(1+b/a)M_h\theta^{-1}e^{-\rho t}V(z).
\end{aligned}
\]
This proves \eqref{eq:weighted-ergodicity}.
\end{proof}

\subsection{A coefficient criterion}\label{subsec:coefficient-criterion}
The following coefficient conditions verify the path hypotheses
of Theorem~\ref{thm:controlled-ergodicity}.
They also give a minorization time independent of the starting compact set.
Fix a finite nonempty set $\mathcal J\subset\Sset$.

\begin{assumpt}[Coefficient representation]\label{ass:coefficient-representation}
For every $i\in\mathcal J$, there are constant matrices
$A_i\in\R^{d\times d}$, $B_i\in\R^{d\times m}$,
a vector $c_i\in\R^d$, and $C^2$ maps
$f_i:\R^d\to\R^m$, $R_i:\R^d\to\R^{m\times m}$ such that
\begin{equation}\label{ext:eq:feedbackfields}
b_i^\circ(x)=A_ix+c_i+\sigma_i(x)f_i(x),\qquad
\sigma_i(x)R_i(x)=B_i,\qquad x\in\R^d.
\end{equation}
\end{assumpt}

\begin{remark*}[Coefficient representation]
Assumption~\ref{ass:coefficient-representation} allows affine controlled paths to be realized with the original
coefficients. For $x\in\R^d$, a prescribed switching schedule
$\lambda$ on $[0,T]$ taking values in $\mathcal J$, and
$v\in L^2([0,T];\R^m)$, set
\[
\begin{aligned}
\dot\phi_t&=A_{\lambda_t}\phi_t+c_{\lambda_t}+B_{\lambda_t}v_t,
\qquad \phi_0=x,\\
u_t&=R_{\lambda_t}(\phi_t)v_t-f_{\lambda_t}(\phi_t).
\end{aligned}
\]
Then $u\in L^2([0,T];\R^m)$ and
$\dot\phi_t=b_{\lambda_t}^\circ(\phi_t)+\sigma_{\lambda_t}(\phi_t)u_t$.
This reduction lets Proposition~\ref{prop:regular-word} adapt the
switched-linear endpoint construction of
\citet[Section~4.2, (25)--(33)]{SunGeLee2002}.
The identity $\sigma_iR_i=B_i$ means that the fixed directions in
$\Ran B_i$ have smooth control representations through $\sigma_i(x)$;
additional noise directions can vary with $x$.
\end{remark*}

Let $\mathcal G$ have vertex set $\Sset$ and an edge $i\to j$ when
$i\ne j$ and $q_{ij}(x)>0$ for every $x\in\R^d$; write
$\mathcal G[\mathcal J]$ for its induced graph on $\mathcal J$.
For the matrices in Assumption~\ref{ass:coefficient-representation} and
a nonempty subset $\mathcal R\subset\mathcal J$, define
\begin{equation}\label{eq:rank-iteration}
\mathcal V_0=\sum_{i\in\mathcal R}\Ran B_i,\qquad
\mathcal V_{k+1}=\mathcal V_k+\sum_{i\in\mathcal R}A_i\mathcal V_k.
\end{equation}
Here $\Ran B_i$ denotes the column space of $B_i$.

\begin{assumpt}[Accessibility and rank]\label{ass:smoothing}
\leavevmode
\begin{enumerate}[label=\textup{(\roman*)},leftmargin=*]
\item \emph{Accessibility.} The graph $\mathcal G[\mathcal J]$ is
strongly connected, and every $i\in\Sset$ has a finite directed path
in $\mathcal G$ to $\mathcal J$.
\item \emph{Rank across regimes.} There is a nonempty
$\mathcal R\subset\mathcal J$ for which
\eqref{eq:rank-iteration} satisfies $\mathcal V_d=\R^d$.
\end{enumerate}
\end{assumpt}

\begin{remark*}[Rank and accessibility]
The rank iteration is the switched-linear controllability criterion
of \citet[Theorem~1, (22), and Section~4.3, (35)]{SunGeLee2002}; see also
\citet[Lemma~1 and Remark~1]{LiuLinChen2013}.
Proposition~\ref{prop:regular-word} inserts admissible transitions
as in the graph-constrained setting of
\citet[Theorem~3.1]{Sun2007Constrained}.

Item~(i) is motivated by the positive-rate connecting paths in
\citet[Assumption~6.2]{XiZhu2017} and
\citet[Assumption~3.3(iii)]{Kunwai2022}, where each selected rate is
positive on a set of positive Lebesgue measure. Here positivity for
every $x$ along selected edges gives positive lower bounds on compact
path neighborhoods. Strong connectivity of $\mathcal G[\mathcal J]$
joins the controllability regimes, while access from every regime to
$\mathcal J$ supplies the entrance paths.
\end{remark*}

	For a reference regime $\ell_*$ and a bounded open ball
$B_*\subset\R^d$, define the target probability measure
\begin{equation}\label{eq:common-target-probability}
\nu(A)=\frac{1}{\Leb(B_*)}
\int_{B_*}\one_A(y,\ell_*)\,dy,
\qquad A\in\B(E).
\end{equation}

\begin{thm}[Common-time minorization]\label{thm:common-time-minorization}
		Suppose Assumptions~\ref{ass:local-construction},
        \ref{ass:coefficient-representation}, and~\ref{ass:smoothing}
		hold and the switching diffusion is nonexplosive.
		Fix $p_*\in\R^d$ and $\ell_*\in\mathcal J$. There are a ball $B_*$
		about $p_*$ and a time $h>0$ such that, for every compact
		$K\subset\R^d$ and finite nonempty $F\subset\Sset$, some
		$\varepsilon_{K,F}>0$ satisfies
		\begin{equation}\label{eq:coefficient-minorization}
			P_h(z,\cdot)\ge\varepsilon_{K,F}\nu(\cdot),
			\qquad z\in K\times F,
		\end{equation}
		where $\nu$ is the probability measure in
		\eqref{eq:common-target-probability}. 
	\end{thm}

	\begin{cor}[Ergodicity under the coefficient criterion]
\label{thm:coefficient-ergodicity}
Under Assumptions~\ref{ass:local-construction},
\ref{ass:joint-lyapunov}, \ref{ass:coefficient-representation}, and
\ref{ass:smoothing}, the joint process is nonexplosive and has a
unique invariant probability $\pi$, with $\pi(V)\le b/a$, and
satisfies \eqref{eq:weighted-ergodicity}.
\end{cor}
\begin{proof}
The stopped Lyapunov argument in the proof of
Theorem~\ref{thm:controlled-ergodicity}, up to
\eqref{eq:joint-semigroup-drift}, uses only
Assumptions~\ref{ass:local-construction} and~\ref{ass:joint-lyapunov}.
It gives nonexplosion and that drift estimate. Thus
Theorem~\ref{thm:common-time-minorization} applies.
Choose $R>2b/a$ with $C_R\ne\varnothing$. Properness puts $C_R$
in some $K\times F$ with $K$ compact and $F$ finite, so its
$h$-step kernel satisfies minorization and
\[
P_hV\le\gamma V+K_h,\qquad
\gamma=e^{-ah},\quad K_h=\frac ba(1-\gamma),\quad
\frac{2K_h}{1-\gamma}=\frac{2b}{a}<R.
\]
The Harris argument in the proof of
Theorem~\ref{thm:controlled-ergodicity}, starting with this
$h$-step kernel, gives the stated conclusion.
\end{proof}

The following sufficient condition for Assumption~\ref{ass:joint-lyapunov}
is useful for the application in Section~\ref{sec:examples}.
 A related additive construction for
countably many regimes appears in
\citet[Theorem~4.1 and its proof]{Kunwai2022}.
Choose symmetric matrices $P_i\in\R^{d\times d}$ satisfying
\begin{equation}\label{eq:matrix-bounds}
\underline p I\preceq P_i\preceq\overline p I,
\qquad 0<\underline p\le\overline p<\infty,
\end{equation}
where $\preceq$ is the quadratic-form order, and put
\[
\begin{aligned}
QP(x,i)&=\sum_{j\ne i}q_{ij}(x)(P_j-P_i),\\
D_P(x,i)&=2x^\top P_ib_i(x)
+\operatorname{tr}(\sigma_i(x)^\top P_i\sigma_i(x))
+x^\top QP(x,i)x.
\end{aligned}
\]
The matrix sum converges absolutely by \eqref{eq:matrix-bounds}
and $q_i<\infty$.

\begin{proposition*}[Quadratic Lyapunov criterion]
Under Assumption~\ref{ass:local-construction}(ii), suppose
\eqref{eq:matrix-bounds} holds. Let $\theta>0$ and let
$\chi:\Sset\to[1,\infty)$ have finite sublevel sets, with
$\sum_{j\ne i}q_{ij}(x)|\chi(j)-\chi(i)|$ locally bounded.
If, for some $\eta,\gamma>0$ and $K_0\ge0$,
\begin{equation}\label{eq:joint-quadratic-certificate}
D_P(x,i)+\theta Q\chi(x,i)
\le-\eta|x|^2-\gamma\chi(i)+K_0\quad\text{on }E,
\end{equation}
then $V(x,i)=1+x^\top P_ix+\theta\chi(i)$ satisfies
Assumption~\ref{ass:joint-lyapunov}, with
$a=\min\{\eta/\overline p,\gamma/\theta\}$ and $b=K_0+a$.
\end{proposition*}
\begin{proof}
The generator and absolute jump sum satisfy
\[
\begin{aligned}
\A V&=D_P+\theta Q\chi
\le-\eta|x|^2-\gamma\chi(i)+K_0\le-aV+b,\\
J_V(x,i)&\le2\overline p|x|^2q_i(x)
+\theta\sum_{j\ne i}q_{ij}(x)|\chi(j)-\chi(i)|.
\end{aligned}
\]
The second bound is locally bounded by the hypotheses. Finally,
$V\ge1+\underline p|x|^2+\theta\chi(i)$ and continuity of
$V(\cdot,i)$ give compact sublevel sets.
\end{proof}

\section{Minorization from prescribed switching schedules}\label{sec:rank-cycle}

\subsection{Prescribed switching schedules and controlled endpoints}
\label{subsec:controlled-endpoints}
Use the prescribed switching schedules defined in
Section~\ref{subsec:controlled-paths} and the controlled equation
\eqref{eq:skeleton}. Let $\mathcal R\subset\mathcal J$ be as in
Assumption~\ref{ass:smoothing}(ii), and write $T_\lambda$ for the
duration of a prescribed switching schedule $\lambda$.
For prescribed switching schedules $\lambda,\eta$ taking values in
$\mathcal J$, define their concatenation by
\[
(\lambda\star\eta)_t=
\begin{cases}
\lambda_t,&0\le t<T_\lambda,\\
\eta_{t-T_\lambda},&T_\lambda\le t\le T_\lambda+T_\eta,
\end{cases}
\qquad T_{\lambda\star\eta}=T_\lambda+T_\eta.
\]
Equal regimes at the join form one visit. For $i\in\mathcal J$ and
$t>0$, write $\lambda^{i,t}$ for the constant prescribed switching
schedule with $\lambda^{i,t}_s=i$ for $0\le s\le t$.

For the original coefficients, Assumption~\ref{ass:coefficient-representation}
gives the reduction
\[
u_t=R_{\lambda_t}(\phi_t)v_t-f_{\lambda_t}(\phi_t)
\quad\Longrightarrow\quad
\dot\phi_t=A_{\lambda_t}\phi_t+c_{\lambda_t}+B_{\lambda_t}v_t.
\]
Let $\Phi_\lambda(t,s)$ solve
$\partial_t\Phi_\lambda(t,s)=A_{\lambda_t}\Phi_\lambda(t,s)$,
$\Phi_\lambda(s,s)=I$. Variation of constants gives
\begin{equation}\label{eq:affine-data}
\begin{aligned}
m_\lambda(x)&=\Phi_\lambda(T_\lambda,0)x
 +\int_0^{T_\lambda}\Phi_\lambda(T_\lambda,s)c_{\lambda_s}\,ds,\\
K_\lambda v&=\int_0^{T_\lambda}\Phi_\lambda(T_\lambda,s)B_{\lambda_s}v_s\,ds,
\qquad \phi_{T_\lambda}=m_\lambda(x)+K_\lambda v.
\end{aligned}
\end{equation}

\begin{prop}[Controllability with one prescribed switching schedule]
\label{prop:regular-word}
Suppose Assumption~\ref{ass:coefficient-representation} holds and
$b_i^\circ,\sigma_i$ are $C^2$ for every $i\in\mathcal J$.
Assume the rank condition in Assumption~\ref{ass:smoothing}(ii) and
strong connectivity of $\mathcal G[\mathcal J]$ as in item~(i).
Fix $\ell_*\in\mathcal J$.
There exist $T_0>0$ and a prescribed switching schedule $\lambda^*$
of duration $T_{\lambda^*}=T_0$, taking values in $\mathcal J$, with
$\lambda^*_0=\lambda^*_{T_0}=\ell_*$ and successive transitions
along edges of $\mathcal G[\mathcal J]$, such that
$K_{\lambda^*}:L^2([0,T_0];\R^m)\to\R^d$ is onto.
For every $x,y\in\R^d$, there is a deterministic control $u\in L^2([0,T_0];\R^m)$ satisfying
\[
\mathcal E_{x,T_0,\lambda^*}(u)=y,\qquad
D_u\mathcal E_{x,T_0,\lambda^*}(u):L^2([0,T_0];\R^m)\to\R^d
\quad\hbox{is onto}.
\]
\end{prop}
\begin{proof}
\emph{1. Use the rank condition to construct a prescribed switching schedule without graph constraints.}
The linear controllability conclusion for one prescribed switching
schedule is classical; see
\citet[Theorem~1 and (22)]{SunGeLee2002} and
\citet[Lemma~1 and Remark~1]{LiuLinChen2013}.
We prove it by concatenation. Let $\Omega_{\mathcal R}$ be the set of
all prescribed switching schedules taking values in $\mathcal R$,
of arbitrary positive duration and without graph constraints, and put
\[
L_\lambda=\Ran K_\lambda,\qquad
\mathcal W=\operatorname{span}
 \left(\bigcup_{\lambda\in\Omega_{\mathcal R}}L_\lambda\right).
\]
The stabilized space $\mathcal V_d$ in
\eqref{eq:rank-iteration} is the smallest common $A_i$-invariant
subspace containing $\Ran B_i$ for $i\in\mathcal R$.
Consequently, \eqref{eq:affine-data} gives
$\mathcal W\subset\mathcal V_d$. For $i\in\mathcal R$, $t>0$, and
$z\in\R^d$,
\[
z\perp L_{\lambda^{i,t}}
\ \Longrightarrow\
\int_0^t|B_i^\top e^{A_i^\top(t-s)}z|^2\,ds=0
\ \Longrightarrow\ B_i^\top z=0.
\]
The continuous nonnegative integrand vanishes at $s=t$, so taking
orthogonal complements yields
$\Ran B_i\subset L_{\lambda^{i,t}}\subset\mathcal W$.
Appending a visit gives
\[
L_{\lambda\star\lambda^{i,t}}=e^{tA_i}L_\lambda+L_{\lambda^{i,t}}
\subset\mathcal W.
\]
Hence $e^{tA_i}\mathcal W\subset\mathcal W$, and closedness gives
$A_iw=\lim_{t\downarrow0}(e^{tA_i}w-w)/t\in\mathcal W$ for
$w\in\mathcal W$. Thus $\mathcal W=\mathcal V_d=\R^d$, where the
last equality uses Assumption~\ref{ass:smoothing}(ii).

Choose $\lambda\in\Omega_{\mathcal R}$ maximizing $\dim L_\lambda$.
For every $\eta\in\Omega_{\mathcal R}$, concatenation gives
\[
L_{\eta\star\lambda}
 =\Phi_\lambda(T_\lambda,0)L_\eta+L_\lambda=L_\lambda,
\]
where the second equality follows from maximality. Since
$\Phi_\lambda(T_\lambda,0)$ is invertible,
\[
\R^d=\Phi_\lambda(T_\lambda,0)\mathcal W
 =\operatorname{span}_{\eta\in\Omega_{\mathcal R}}
   \Phi_\lambda(T_\lambda,0)L_\eta
 \subset L_\lambda\subset\R^d.
\]
Applied to this affine endpoint map,
Lemma~\ref{lem:endpoint-directions} supplies orthonormal step
functions $v_1,\ldots,v_d$ such that the matrix
$H=[K_\lambda v_1\ \cdots\ K_\lambda v_d]\in\R^{d\times d}$ is invertible.
Its $j$th column is $K_\lambda v_j$.

\emph{2. Use connectivity to insert admissible transitions.}
For graph-constrained reachability, see
\citet[Theorem~3.1]{Sun2007Constrained}.
Write the selected prescribed switching schedule as
$\lambda=\lambda^{\boldsymbol\iota,\tau}$.
Assumption~\ref{ass:smoothing}(i) supplies directed paths joining
$\ell_*,\iota_0,\ldots,\iota_n,\ell_*$. Retain each original visit and give
each inserted interval length $\epsilon>0$, obtaining a prescribed switching schedule
$\lambda^\epsilon$ along graph edges, with initial and terminal regime
$\ell_*$ and duration $T_\epsilon=T_\lambda+N\epsilon$, where $N$
counts the inserted intervals. Keep the original and inserted intervals
marked, including when adjacent equal regimes combine into one visit.

Let $N_r$ count inserted intervals before original visit $r$,
$0\le r\le n$. On $[\tau_r,\tau_{r+1})$ put
$\theta_\epsilon(s)=s+N_r\epsilon$ and translate the directions by
\[
\begin{aligned}
v_j^\epsilon(\theta_\epsilon(s))&=v_j(s),\qquad
v_j^\epsilon=0\quad\hbox{on inserted intervals},\\
\int_0^{T_\epsilon}v_j^\epsilon(t)\cdot v_l^\epsilon(t)\,dt
 &=\sum_{r=0}^n\int_{\tau_r}^{\tau_{r+1}}v_j(s)\cdot v_l(s)\,ds
 =\delta_{jl}.
\end{aligned}
\]
In the finite product for the propagator, insertion adds factors
$e^{\epsilon A_i}$ with
$\|e^{\epsilon A_i}-I\|\le\epsilon\|A_i\|e^{\epsilon\|A_i\|}$,
where $\|\cdot\|$ is the Euclidean operator norm.
Thus \eqref{eq:affine-data} gives
\[
H_\epsilon=[K_{\lambda^\epsilon}v_1^\epsilon\ \cdots\
K_{\lambda^\epsilon}v_d^\epsilon],\qquad
\|H_\epsilon-H\|\le C\epsilon,
\quad 0<\epsilon\le1.
\]
Choose $\epsilon$ with
$\|H^{-1}\|\|H_\epsilon-H\|<1/2$; the Neumann series then gives
$\|H_\epsilon^{-1}\|\le2\|H^{-1}\|$.
Fix $\epsilon$ as above and put $\lambda^*=\lambda^\epsilon$ and
$T_0=T_{\lambda^*}=T_\epsilon$. Relabel $v_j^\epsilon,H_\epsilon$
as $v_j,H$.

\emph{3. Lift the affine controls to the original equation.}
Following the linear endpoint construction in
\citet[Section~4.2, (25)--(33)]{SunGeLee2002}, set
\[
v^0=\sum_{j=1}^d[H^{-1}(y-m_{\lambda^*}(x))]_jv_j,\qquad
v^\xi=v^0+\sum_{j=1}^d\xi_jv_j.
\]
The affine path $\phi^\xi$ has endpoint $y+H\xi$ and depends affinely
on $\xi$. For $\xi$ in a small closed ball, these paths lie in a
common compact set. Assumption~\ref{ass:coefficient-representation} gives the
$C^2$ maps $f_i,R_i$ and the path-preserving controls
\[
u^\xi(t)=R_{\lambda^*_t}(\phi_t^\xi)v^\xi(t)
-f_{\lambda^*_t}(\phi_t^\xi),\qquad
\mathcal E_{x,T_0,\lambda^*}(u^\xi)=y+H\xi.
\]
Since $v^0,v_j$ are step functions, the controls $u^\xi$ and their
first two parameter derivatives are uniformly bounded on this compact
family. Taylor's formula therefore gives, for $\xi,\xi+\eta$ in the ball,
\[
\begin{aligned}
\|u^{\xi+\eta}-u^\xi-D_\xi u^\xi\eta\|_{L^2}
 &\le C|\eta|^2,\\
\|D_\xi u^{\xi+\eta}-D_\xi u^\xi\|_{\R^d\to L^2}
 &\le C|\eta|.
\end{aligned}
\]
Thus $\xi\mapsto u^\xi$ is $C^1$ into $L^2$.
The assumed $C^2$ regularity of $b_i^\circ,\sigma_i$ on the
regimes visited by $\lambda^*$ allows us to apply
Lemma~\ref{lem:endpoint-directions}; the chain rule gives
\[
D_u\mathcal E_{x,T_0,\lambda^*}(u^0)
\left.D_\xi u^\xi\right|_{\xi=0}=H.
\]
Since $H$ is invertible, the endpoint derivative is onto.
\end{proof}

\subsection{From controlled endpoints to local minorization}
\label{subsec:local-minorization}
We prove Theorem~\ref{prop:local}.  Fix $T>0$ and a regime sequence
$\boldsymbol\iota$ as in Section~\ref{subsec:controlled-paths}.

Fix a compact set \(K^*\subset\R^d\) and globally Lipschitz
extensions of \(b_i,\sigma_i\) agreeing with the original
coefficients near \(K^*\). Let \(Y^{x,\boldsymbol\iota,\tau}\) solve
\begin{equation}\label{eq:prescribed-sde}
\begin{aligned}
dY_t^{x,\boldsymbol\iota,\tau}
&=b_{\lambda_t^{\boldsymbol\iota,\tau}}
  (Y_t^{x,\boldsymbol\iota,\tau})\,dt
 +\sigma_{\lambda_t^{\boldsymbol\iota,\tau}}
  (Y_t^{x,\boldsymbol\iota,\tau})\,dW_t,\\
Y_0^{x,\boldsymbol\iota,\tau}&=x.
\end{aligned}
\end{equation}
The law restricted to \(C([0,T];K^*)\) is independent of these extensions.

For a continuous path \(y\), set
\begin{equation}\label{eq:path-weight}
L_{T,\tau}(y)=
\exp\left\{-\int_0^Tq_{\lambda_s^{\boldsymbol\iota,\tau}}(y_s)\,ds\right\}
\prod_{r=1}^nq_{\iota_{r-1}\iota_r}(y_{\tau_r}).
\end{equation}

The next lemma records a localized consequence of the standard
interlacing construction.

	\begin{lem}[Localized switching-path lower bound]
		\label{lem:path-expansion}
		Under Assumption~\ref{ass:local-construction}, let
		$\mathcal D\subset\Delta_n(T)$ be a Borel set of switching times and
		$\mathcal T_\tau\subset C([0,T];K^*)$ a Borel path restriction
		measurable jointly in $\tau$ and the path. Suppose that for every
		$\tau\in\mathcal D$ and $y\in\mathcal T_\tau$,
		\[
		\begin{aligned}
		q_{\lambda_s^{\boldsymbol\iota,\tau}}(y_s)&\le Q_*<\infty,
		&&0\le s\le T,\\
		q_{\iota_{r-1}\iota_r}(y_{\tau_r})&\ge q_*>0,
		&&1\le r\le n.
		\end{aligned}
		\]
		Then, for $D\in\B(\R^d)$,
		\begin{equation}\label{eq:word-lower-bound}
			\begin{aligned}
				P_T((x,\iota_0),D\times\{\iota_n\})
				&\ge\int_{\mathcal D}\E[
				\one_{\{Y_T^{x,\boldsymbol\iota,\tau}\in D,\,
					Y^{x,\boldsymbol\iota,\tau}\in\mathcal T_\tau\}}
				L_{T,\tau}(Y^{x,\boldsymbol\iota,\tau})]\,d\tau\\
				&\ge e^{-Q_*T}q_*^n\int_{\mathcal D}
				\PP\{Y_T^{x,\boldsymbol\iota,\tau}\in D,\,
				Y^{x,\boldsymbol\iota,\tau}\in\mathcal T_\tau\}\,d\tau.
			\end{aligned}
		\end{equation}
	\end{lem}
	\begin{proof}
Let $\mathcal J_{\boldsymbol\iota}=\{\iota_0,\ldots,\iota_n\}$.
Choose a bounded open neighborhood $O$ of $K^*$ whose compact closure
lies where the visited coefficient extensions agree with the originals.
Let $T_r$ be the regime jump times, with $T_0=0$ and absent jumps set to
$\infty$. Define
\[
\begin{aligned}
\eta&=\inf\{t\ge0:X_t\notin O\ \text{or}\
\Lambda_t\notin\mathcal J_{\boldsymbol\iota}\},\\
\mathcal H_{\boldsymbol\iota,T}
&=\{T<\eta,\ T_n<T<T_{n+1},\
\Lambda_{T_r}=\iota_r\ (1\le r\le n)\}.
\end{aligned}
\]
Assumption~\ref{ass:local-construction} gives unique frozen diffusions
and bounded total rates on $\overline O\times\mathcal J_{\boldsymbol\iota}$,
so this stopped construction has finitely many jumps before $T\wedge\eta$.
The exponential-clock construction of
\citet[proof of Theorem~2.1, Step~1, (2.2)--(2.5)]{XiZhu2017}
gives the first-switch density
$e^{-\int_0^tq_i(Y_s)\,ds}q_{ij}(Y_t)\,dt$ along a frozen
$i$-diffusion $Y$, before spatial exit. Iterating it and including
survival after the last switch and applying Tonelli's theorem gives,
for bounded nonnegative cylinder
functionals $F(y,\tau)$,
\[
\begin{aligned}
&\E_{(x,\iota_0)}[F(X,(T_1,\ldots,T_n));
\mathcal H_{\boldsymbol\iota,T}]\\
&\quad=\int_{\Delta_n(T)}\E\Bigl[
F(Y^{x,\boldsymbol\iota,\tau},\tau)
L_{T,\tau}(Y^{x,\boldsymbol\iota,\tau})\,
\one_{\{Y^{x,\boldsymbol\iota,\tau}\in C([0,T];O)\}}\Bigr]\,d\tau.
\end{aligned}
\]
The jointly Borel prescribed solution kernels and bounded rates make
the two sides finite measures on
$C([0,T];\R^d)\times\Delta_n(T)$. The monotone-class theorem extends
the identity from cylinder sets to Borel sets; monotone convergence
then gives it for every nonnegative Borel $F$.
Taking $F(y,\tau)=\one_{\mathcal D}(\tau)\one_D(y_T)
\one_{\mathcal T_\tau}(y)$ removes the spatial exit indicator and
gives the first inequality in \eqref{eq:word-lower-bound}.
For $\tau\in\mathcal D$ and $y\in\mathcal T_\tau$, the rate bounds
in the statement give
\[
\int_0^Tq_{\lambda_s^{\boldsymbol\iota,\tau}}(y_s)\,ds\le Q_*T,
\qquad
\prod_{r=1}^nq_{\iota_{r-1}\iota_r}(y_{\tau_r})\ge q_*^n.
\]
Hence $L_{T,\tau}(y)\ge e^{-Q_*T}q_*^n$, which gives the second inequality.
\end{proof}

We now use an admissible regular control to bound the joint
probability in \eqref{eq:word-lower-bound}. For this argument, fix
$T$ and a finite regime sequence $\boldsymbol\iota$. In the shifted family below, use
compactly supported localizations of the visited It\^o coefficients,
formed with smooth cutoffs, and their corresponding Stratonovich drifts.
The localized coefficients satisfy $b_i\in C_c^{d+3}$ and
$\sigma_i\in C_c^{d+4}$ by Assumption~\ref{ass:local-construction}(i).

Write $\alpha=(x,\tau)\in\R^d\times\Delta_n(T)$ for the initial-point
and switching-time parameter. The prescribed switching schedule
$\lambda^{\boldsymbol\iota,\tau}:[0,T]\to\Sset$, defined in
Section~\ref{sec:model}, depends only on $\tau$.

Use canonical Wiener space $C_0([0,T];\R^m)$, the continuous paths
$w$ with $w(0)=0$, with its uniform norm and Borel $\sigma$-field.
Let $\PP$ be Wiener measure and $W_t(w)=w(t)$ the coordinate process.
The Cameron--Martin space and its inner product are
\[
\begin{aligned}
\mathcal H
&=\{h\in C_0([0,T];\R^m):h\text{ is absolutely continuous},
\ \dot h\in L^2([0,T];\R^m)\},\\
\langle h,g\rangle_{\mathcal H}
&=\int_0^T\dot h(t)\cdot\dot g(t)\,dt.
\end{aligned}
\]
The dot denotes the Euclidean inner product. For $e_1,\ldots,e_d\in
\mathcal H$ with step-function derivatives, set
\[
a(t)=[\dot e_1(t)\ \cdots\ \dot e_d(t)]\in\R^{m\times d},\qquad
h_\xi(t)=\sum_{j=1}^d\xi_je_j(t)=\int_0^t a(r)\xi\,dr.
\]
For $u\in L^2$, the integrated control $h(t)=\int_0^tu_r\,dr$
also belongs to $\mathcal H$. Fix a closed ball $Q\subset\R^d$ of
positive radius and a bounded open neighborhood $U_Q$ of $Q\cup\{0\}$.
For $\xi\in U_Q$, define the shifted SDE family by
\begin{equation}\label{eq:shifted-sde}
\begin{aligned}
dY_t^\alpha(\xi)
&=\bigl[b_{\lambda_t^{\boldsymbol\iota,\tau}}^\circ(Y_t^\alpha(\xi))
  +\sigma_{\lambda_t^{\boldsymbol\iota,\tau}}(Y_t^\alpha(\xi))a(t)\xi\bigr]dt\\
&\quad+\sigma_{\lambda_t^{\boldsymbol\iota,\tau}}(Y_t^\alpha(\xi))
  \circ dW_t,\qquad Y_0^\alpha(\xi)=x.
\end{aligned}
\end{equation}
Thus $Y^\alpha(0)=Y^{x,\boldsymbol\iota,\tau}$ with the same
localized coefficients. For a deterministic input $w$, write
$Y^\alpha(\xi;w)$ for the path-valued version supplied by
Lemma~\ref{lem:classical-family}; $Y^\alpha(\xi;W)$ is the random path.
Its endpoint map is $G_{\alpha,W}(\xi)=Y_T^\alpha(\xi)$.
The derivative $D_\xi G_{\alpha,W}$ acts on $\R^d$, whereas
$D_u\mathcal E_{x,T,\lambda}$ acts on $L^2$ controls.
To transfer both paths and parameter derivatives by support, for a
$C^1$ family of paths $F$ use
\begin{equation}\label{eq:family-norm}
\|F\|_{Q,1}
=\sup_{\xi\in Q,\,t\le T}|F_t(\xi)|
 +\sup_{\xi\in Q,\,t\le T}\|D_\xi F_t(\xi)\|.
\end{equation}
Lemma~\ref{lem:classical-family} provides support in
$\|\cdot\|_{Q,1}$ and continuity in the initial point and switching
times. Support for one trajectory alone would not control
$D_\xi G_{\alpha,W}$, which is needed for a common image ball and
Jacobian bound in the averaging argument.

\begin{proof}[Proof of Theorem~\ref{prop:local}]
We construct a controlled family, transfer its estimates to the
shifted diffusion, and apply Lemma~\ref{lem:averaging} together with
the switching-path bound.

\emph{1. Construct the deterministic reference and its estimates.}
By \eqref{eq:regular-controlled-path},
\[
\phi_0=p,\qquad \phi_{T_0}=y_*,\qquad
D_u\mathcal E_{p,T_0,\lambda}(u_*):L^2\longrightarrow\R^d
\quad\hbox{is onto}.
\]
By \eqref{eq:admissible-controlled-path},
\begin{equation}\label{eq:local-path-admissibility}
q_{\iota_{r-1}\iota_r}(\phi_{\tau_r^*})>0,
\qquad 1\le r\le n.
\end{equation}
Choose a compact set $K^*$, an open set $V$, and a cutoff $\chi$ with
\[
\phi([0,T_0])\subset\operatorname{int}K^*,\qquad
K^*\subset V,\qquad
\chi\in C_c^\infty(\R^d;[0,1]),\qquad \chi|_V=1.
\]
Use $(\chi b_i,\chi\sigma_i)$, $i\in\{\iota_0,\ldots,\iota_n\}$,
as the localized It\^o coefficients in \eqref{eq:shifted-sde}.
Their Stratonovich drifts agree with the originals on $V$.
For $n\ge1$, continuity and \eqref{eq:local-path-admissibility}
give $\rho_*>0$ and $q_*>0$ such that
\[
q_{\iota_{r-1}\iota_r}(z)\ge q_*
\quad\text{if }|z-\phi_{\tau_r^*}|\le\rho_*,
\qquad 1\le r\le n.
\]
For $n=0$, take $\rho_*=q_*=1$.
Assumption~\ref{ass:local-construction}(ii) gives
\[
Q_*:=\sup_{\substack{z\in K^*\\i\text{ visited}}}q_i(z)<\infty.
\]
Uniform continuity of $\phi$ permits a choice of $\delta>0$ and a
switching-time box $\mathcal D\ni\tau^*$, with
$\overline{\mathcal D}\subset\Delta_n(T_0)$, such that
\begin{equation}\label{eq:local-uniform-path-tube}
\mathcal T=\{y\in C([0,T_0];\R^d):\|y-\phi\|_\infty<\delta\}
\end{equation}
is contained in $C([0,T_0];K^*)$ and
\[
\delta+\sup_{\tau\in\mathcal D}
|\phi_{\tau_r}-\phi_{\tau_r^*}|<\rho_*,\qquad 1\le r\le n.
\]
For $y\in\mathcal T$ and $\tau\in\mathcal D$,
$|y_{\tau_r}-\phi_{\tau_r^*}|<\rho_*$, so the two rate bounds
of Lemma~\ref{lem:path-expansion} hold uniformly.
For $n=0$, the switching-time box has volume one.

Lemma~\ref{lem:endpoint-directions} selects the directions $e_j$
in the preceding setup so that
\[
\langle e_j,e_k\rangle_{\mathcal H}=\delta_{jk},\qquad
H_0:=\bigl[D_u\mathcal E_{p,T_0,\lambda}(u_*)\dot e_1\ \cdots\
D_u\mathcal E_{p,T_0,\lambda}(u_*)\dot e_d\bigr]\in\R^{d\times d}
\quad\text{is invertible}.
\]
Let $\phi^0(\xi)$ be the controlled path from $p$ with control
$u_*+a\xi$, and set
\[
\begin{gathered}
u^\xi(t)=u_*(t)+a(t)\xi,\qquad \phi_0^0(\xi)=p,\qquad
\phi^0(0)=\phi,\\
\phi^0_{T_0}(\xi)=\mathcal E_{p,T_0,\lambda}(u^\xi),\qquad
\phi^0_{T_0}(0)=y_*,\qquad D_\xi\phi^0_{T_0}(0)=H_0.
\end{gathered}
\]
The trajectory continuity and continuous endpoint differential in
Lemma~\ref{lem:endpoint-directions} give $r>0$ such that, for
$Q=\overline B(0,2r)$,
\[
\sup_{\xi\in Q}\|\phi^0(\xi)-\phi\|_\infty<\delta/3,\qquad
\sup_{\xi\in Q}\|H_0^{-1}(D_\xi\phi^0_{T_0}(\xi)-H_0)\|<1/4.
\]

\emph{2. Transfer the estimates and verify the inverse-lemma hypotheses.}
Assumption~\ref{ass:local-construction}(i) gives
$\chi b_i\in C_c^{d+3}$ and $\chi\sigma_i\in C_c^{d+4}$ for every
visited regime, so Lemma~\ref{lem:classical-family} applies to the
localized shifted family on $[0,T_0]$. For $\alpha=(x,\tau)$ and
$\alpha_*=(p,\tau^*)$, write
\[
G(\xi)=Y^\alpha_{T_0}(\xi),\qquad I=B(0,r).
\]
Part~\textup{(i)} supplies jointly measurable versions with $C^1$
endpoint maps on the common full-measure set $\Omega_*$.
Choose
\[
0<\eta<\min\left\{\delta/3,
\frac{r}{8\|H_0^{-1}\|},\frac{1}{4\|H_0^{-1}\|}\right\}.
\]
Part~\textup{(ii)}, with $u_0=u_*$, gives support at $\alpha_*$;
part~\textup{(iii)} gives continuity in probability in the norm
\eqref{eq:family-norm}. Thus choose a neighborhood $U\ni p$ and
shrink $\mathcal D\ni\tau^*$ so that
\[
\begin{aligned}
&q_0:=\PP\{\|Y^{\alpha_*}-\phi^0\|_{Q,1}<\eta/2\}>0,\\
&\sup_{\alpha\in U\times\mathcal D}
\PP\{\|Y^\alpha-Y^{\alpha_*}\|_{Q,1}\ge\eta/2\}<q_0/2.
\end{aligned}
\]
For the measurable events
$E_\alpha=\Omega_*\cap\{\|Y^\alpha-\phi^0\|_{Q,1}<\eta\}$,
the triangle inequality gives
\[
\PP(E_\alpha)\ge q_0-
\PP\{\|Y^\alpha-Y^{\alpha_*}\|_{Q,1}\ge\eta/2\}
\ge q_0/2,\qquad \alpha\in U\times\mathcal D.
\]
On $E_\alpha$, the deterministic estimates from Step~1 yield
\[
\begin{aligned}
\sup_{\xi\in Q}\|Y^\alpha(\xi)-\phi\|_\infty
&<\eta+\delta/3<\delta,\\
|G(0)-y_*|&<\eta<\frac{r}{8\|H_0^{-1}\|},\\
\sup_{\xi\in Q}\|H_0^{-1}(D_\xi G(\xi)-H_0)\|
&<\|H_0^{-1}\|\eta+1/4<1/2.
\end{aligned}
\]
The first inequality gives $Y^\alpha(\xi)\in\mathcal T$ for every
$\xi\in Q$. Apply Lemma~\ref{lem:uniform-finite-dimensional-inverse}
with $A=H_0$, $s=\|H_0^{-1}\|^{-1}$, $\xi_0=0$, $\rho=r$, and
$\kappa=1/2$. Then $G|_I$ is a $C^1$ diffeomorphism and
\[
\begin{gathered}
B_*:=B\!\left(y_*,\frac{r}{4\|H_0^{-1}\|}\right)
\subset B\!\left(G(0),\frac{r}{2\|H_0^{-1}\|}\right)\subset G(I),\\
\sup_{\xi\in I}|\det D_\xi G(\xi)|\le(3\|H_0\|/2)^d=:J.
\end{gathered}
\]

\emph{3. Apply the averaging lemma and the switching likelihood.}
For every $\alpha\in U\times\mathcal D$, the event $E_\alpha$
has probability at least $q_0/2$. On it, the shifted paths belong
to $\mathcal T$, and their endpoint map has common image ball $B_*$
and upper Jacobian bound $J$. Joint measurability and the identity
under Cameron--Martin shifts follow from
Lemma~\ref{lem:classical-family}(i),(iv). Apply Lemma~\ref{lem:averaging}
with $q=q_0/2$, $O=B_*$, and $\mathcal T_\alpha=\mathcal T$ to obtain
a common $c_{\rm den}>0$ such that
\begin{equation}\label{eq:fixed-schedule-density}
\PP\{Y_{T_0}^{x,\boldsymbol\iota,\tau}\in D,
Y^{x,\boldsymbol\iota,\tau}\in\mathcal T\}
\ge c_{\rm den}\Leb(D\cap B_*),
\qquad (x,\tau)\in U\times\mathcal D.
\end{equation}
Since $\Leb_n(\mathcal D)>0$, the switching-path bound of
Lemma~\ref{lem:path-expansion} proves \eqref{eq:local} with
\[
c_*=e^{-Q_*T_0}q_*^n\Leb_n(\mathcal D)c_{\rm den}>0.
\]
\end{proof}

\subsection{Controlled entrance and common-time minorization}
\label{subsec:common-time-access}
The next lemma turns synchronized controlled entrance into a uniform probability bound on a compact set. It supplies the accessibility step in Theorem~\ref{thm:controlled-ergodicity}.

\begin{lem}[Synchronized controlled entrance]
\label{lem:controlled-entrance}
Suppose Assumption~\ref{ass:local-construction} holds.
Fix $T_{\rm a}>0$, a nonempty compact set $C\subset E$, an open
set $U\subset\R^d$, and $\ell_*\in\Sset$.
Suppose that for each $z=(x,i)\in C$ there are a prescribed switching schedule
$\lambda^z$ on $[0,T_{\rm a}]$ and a control
$u^z\in L^2([0,T_{\rm a}];\R^m)$ admissible from $x$ under
$\lambda^z$, with
\[
\lambda^z_0=i,\qquad \lambda^z_{T_{\rm a}}=\ell_*,\qquad
\phi^{x,\lambda^z,u^z}_{T_{\rm a}}\in U.
\]
Then
\[
\inf_{z\in C}P_{T_{\rm a}}(z,U\times\{\ell_*\})>0.
\]
\end{lem}
\begin{proof}
Fix $z=(x,i)\in C$ and choose $\lambda^z,u^z$ as in the
statement. For this initial state, write
$\lambda=\lambda^z=\lambda^{\boldsymbol\iota,\tau^*}$,
$u=u^z$, and $\psi=\phi^{x,\lambda,u}$, where
$\tau^*\in\Delta_n(T_{\rm a})$.
Choose a compact neighborhood $K^*$ of $\psi([0,T_{\rm a}])$ and
localize the visited It\^o coefficients by a smooth cutoff equal
to one near $K^*$. The localized controlled path is still $\psi$.
Continuity of the selected rates at its switching locations gives
$\delta>0$, a switching-time box $\mathcal D\ni\tau^*$ with
$\overline{\mathcal D}\subset\Delta_n(T_{\rm a})$, and $q_*>0$
such that, whenever $\|y-\psi\|_\infty<\delta$ and
$\tau\in\mathcal D$,
\[
y([0,T_{\rm a}])\subset K^*,\qquad y(T_{\rm a})\in U,\qquad
q_{\iota_{r-1}\iota_r}(y_{\tau_r})\ge q_*,\quad 1\le r\le n.
\]
The total rates of the finitely many visited regimes are bounded
on $K^*$ by some $Q_*<\infty$.
Apply Lemma~\ref{lem:classical-family}(ii)--(iii) to the localized
coefficients, with reference control $u$ and
$e_1=\cdots=e_d=0$.
After shrinking $\mathcal D$, these support and continuity statements
give a neighborhood $N_z$ of $x$ and $\beta_z>0$ such that
\[
\inf_{\substack{x'\in N_z\\\tau\in\mathcal D}}
\PP\{\|Y^{x',\boldsymbol\iota,\tau}-\psi\|_\infty<\delta\}
\ge\beta_z.
\]
Lemma~\ref{lem:path-expansion} therefore yields
\[
\inf_{x'\in N_z}P_{T_{\rm a}}((x',i),U\times\{\ell_*\})
\ge e^{-Q_*T_{\rm a}}q_*^n\Leb_n(\mathcal D)\beta_z>0.
\]
For $n=0$, use $q_*=1$ and the existing convention
$\Leb_0(\mathcal D)=1$.
Finitely many sets $N_z\times\{i\}$ cover $C$; the minimum of
their positive lower bounds proves the claim.
\end{proof}

\begin{proof}[Proof of Theorem~\ref{thm:common-time-minorization}]
\emph{1. Fix the local lower bound.}
Proposition~\ref{prop:regular-word} supplies a fixed prescribed switching schedule
$\lambda^*$ of duration $T_0$, starting and ending in $\ell_*$,
and a control $u_*$ whose path returns from $p_*$ to $p_*$
with a surjective endpoint differential. The graph-edge transitions
have positive rates, so $u_*$ is admissible and regular from $p_*$
under $\lambda^*$. Theorem~\ref{prop:local} applies with both
spatial endpoints equal to $p_*$ and both endpoint regimes equal
to $\ell_*$. Write $U,B_*,c_*$ for its local bound.

For $A\in\B(E)$, put
$A_{\ell_*}=\{w\in\R^d:(w,\ell_*)\in A\}$. Then, for $y\in U$,
\[
\begin{aligned}
P_{T_0}((y,\ell_*),A)
&\ge P_{T_0}((y,\ell_*),A_{\ell_*}\times\{\ell_*\})\\
&\ge c_*\Leb(A_{\ell_*}\cap B_*)
=\varepsilon_0\nu(A),\qquad
\varepsilon_0=c_*\Leb(B_*)>0,
\end{aligned}
\]
where $\nu$ is defined by \eqref{eq:common-target-probability}.
Fix these objects before choosing the starting compact set.

\emph{2. Arrange a common arrival time.}
Keep the prescribed switching schedule $\lambda^*$ from Step~1 and fix $(x,i)\in E$.
Assumption~\ref{ass:smoothing}(i) gives graph paths
\[
i=i_0\to\cdots\to i_k=j_0,\qquad
j_0\to\cdots\to j_\ell=\ell_*,\qquad
i_r\notin\mathcal J\ (r<k),\quad j_s\in\mathcal J\ (s\le\ell).
\]
For $k\ge1$, set
\[
M_{x,i}=\max_{0\le r<k}\sup_{|y-x|\le1}|b_{i_r}^\circ(y)|<\infty,
\qquad t_*=\frac1{4(1+M_{x,i})};
\]
for $k=0$, put $t_*=0$ and omit the outside visits.
Let the prescribed switching schedule $\lambda^{x,i}$ visit $i_0,\ldots,i_{k-1}$ for time $t_*/k$ each,
then $j_0,\ldots,j_\ell$ for time $(1-t_*)/(\ell+1)$ each.
Define $\eta^{x,i}_0=x$ and
\[
\dot\eta^{x,i}_t=
\begin{cases}
b_{\lambda^{x,i}_t}^\circ(\eta^{x,i}_t),&0\le t<t_*,\\
A_{\lambda^{x,i}_t}\eta^{x,i}_t+c_{\lambda^{x,i}_t},&t_*\le t\le1.
\end{cases}
\]
For $k\ge1$, up to the first exit from $B(x,1)$, the first part satisfies
$|\eta^{x,i}_t-x|\le M_{x,i}t_*<1/4$, so local ODE continuation
gives existence through $t_*$. The affine part exists through time $1$.
By \eqref{ext:eq:feedbackfields}, this is a controlled path with
\[
w^{x,i}_t=
\begin{cases}
0,&0\le t<t_*,\\
-f_{\lambda^{x,i}_t}(\eta^{x,i}_t),&t_*\le t\le1,
\end{cases}
\qquad w^{x,i}\in L^2([0,1];\R^m),
\]
since the path is compact and the visited $f_j$ are continuous.
Proposition~\ref{prop:regular-word} supplies
$u^{x,i}\in L^2([0,T_0];\R^m)$ and its controlled path
$\phi^{x,i}$ on the prescribed switching schedule $\lambda^*$, with
\[
\phi^{x,i}_0=\eta^{x,i}_1,\qquad
\phi^{x,i}_{T_0}=\mathcal E_{\eta^{x,i}_1,T_0,\lambda^*}(u^{x,i})=p_*.
\]
Append $\lambda^*$ to $\lambda^{x,i}$ at time $1$ and concatenate
the controls $w^{x,i}$ and $u^{x,i}$. With $T_{\rm acc}=1+T_0$, this gives
\[
\psi^{x,i}_t=
\begin{cases}
\eta^{x,i}_t,&0\le t\le1,\\
\phi^{x,i}_{t-1},&1<t\le T_{\rm acc},
\end{cases}
\qquad
\psi^{x,i}_0=x,\quad \psi^{x,i}_{T_{\rm acc}}=p_*.
\]
Combine the two adjoining $\ell_*$ visits at $t=1$.
The resulting prescribed switching schedule has positive visit lengths, graph-edge
transitions, and terminal regime $\ell_*$, at the common time
$T_{\rm acc}=1+T_0$.

\emph{3. Obtain an entrance bound uniform over $K\times F$.}
The case $K\times F=\varnothing$ is vacuous.
Otherwise, $K\times F$ is a nonempty compact subset of $E$.
For each of its points, Step~2 constructs a prescribed switching schedule and an
admissible control whose joint path ends at
$(p_*,\ell_*)\in U\times\{\ell_*\}$ at time $T_{\rm acc}=1+T_0$.
Lemma~\ref{lem:controlled-entrance} therefore gives
$p_{K,F}>0$ such that
\begin{equation}\label{eq:access-bound}
\inf_{x\in K,\,i\in F}
P_{T_{\rm acc}}((x,i),U\times\{\ell_*\})\ge p_{K,F}>0.
\end{equation}
The time and local target neighborhood were fixed before
choosing $K,F$; only the positive lower bound depends on them.

\emph{4. Apply the Markov property.}
Set $h=T_{\rm acc}+T_0=1+2T_0$. For $z\in K\times F$ and $A\in\B(E)$,
\[
\begin{aligned}
P_h(z,A)
&=\int_E P_{T_{\rm acc}}(z,dw)P_{T_0}(w,A)\\
&\ge\int_{U\times\{\ell_*\}}
P_{T_{\rm acc}}(z,dw)P_{T_0}(w,A)\\
&\ge\varepsilon_0\nu(A)
P_{T_{\rm acc}}(z,U\times\{\ell_*\})\\
&\ge p_{K,F}\varepsilon_0\nu(A).
\end{aligned}
\]
Thus \eqref{eq:coefficient-minorization} holds with
$\varepsilon_{K,F}=p_{K,F}\varepsilon_0$.
\end{proof}

\section{Applications and examples}\label{sec:examples}
\subsection{Spatially restricted switching rates}
Let $d=2$, $m=1$, and $\Sset=\Nzero$. Choose
$\chi\in C_c^\infty(\R^2;[0,1])$ equal to one near zero, and set
\[
\begin{gathered}
b_i=0\ (i\ge0),\qquad
\sigma_0=\binom{1}{0},\quad\sigma_1=\binom{0}{1},\quad
\sigma_i=0\ (i\ge2),\\
q_{01}=q_{10}=\chi,\qquad q_{ij}=0\quad\text{for all other }i\ne j.
\end{gathered}
\]
Assumption~\ref{ass:local-construction} holds, and bounded coefficients
and total rates give nonexplosion. For the prescribed switching
schedule $\lambda_t=0$ on $[0,1)$ and $\lambda_t=1$ on $[1,2]$,
the control $u_*=0$ from $p=0$ gives $\phi\equiv0$ and
$q_{01}(\phi_1)=1$. Its endpoint differential is
\[
D_u\mathcal E_{0,2,\lambda}(0)v
=\left(\int_0^1v_s\,ds,\ \int_1^2v_s\,ds\right)^\top.
\]
The images of $\one_{[0,1)}$ and $\one_{[1,2]}$ form the coordinate
basis, so Theorem~\ref{prop:local} gives \eqref{eq:local} with
$T_0=2$, $\iota_0=0$, $\iota_n=1$, and $y_*=0$.
Every frozen law is supported on a line or a point, and no
off-diagonal rate is positive at every spatial point.

\subsection{A network with countable regimes}

Let $\mathcal E$ be the edge set of a connected undirected graph on
$\{1,\ldots,N\}$, where $N\ge3$. Write $\mathbf v_1,\ldots,\mathbf v_N$
for the standard basis of $\R^N$, so $(\mathbf v_j)_k=\one_{\{j=k\}}$.
For each edge $e=\{j,k\}\in\mathcal E$, define its edge Laplacian and
network matrices by
\[
L_e=(\mathbf v_j-\mathbf v_k)(\mathbf v_j-\mathbf v_k)^\top,
\qquad D_e=-a\mathbf v_1\mathbf v_1^\top-\kappa_eL_e,
\qquad A_e=D_e+\delta I_N,
\]
where $a>0$, $\kappa_e>0$, $\delta\ge0$, and $I_N$ is the
$N\times N$ identity matrix. Let $\mathcal I\subset\Nzero$ be a
countably infinite set of noise levels. The remaining coefficients are
$c_e\in\R^N$, $s_n\ge0$, and
$g_e\in C^{N+4}(\R^N;\R)$, with
\[
0<g_-\le g_e(x)\le g_+<\infty,
\qquad \sup_x\bigl(|Dg_e(x)|+\|D^2g_e(x)\|\bigr)<\infty.
\]

The edge and level rate matrices are $Q(x)=(q_{ef}(x))_{e,f\in\mathcal E}$
and $R=(r_{nm})_{n,m\in\mathcal I}$. A constant reference matrix
$Q^0=(q^0_{ef})_{e,f\in\mathcal E}$ will define the quadratic Lyapunov
weights. Their off-diagonal entries are nonnegative, and
\[
q_{ee}(x)=-\sum_{f\ne e}q_{ef}(x),\qquad
q^0_{ee}=-\sum_{f\ne e}q^0_{ef},\qquad
r_{nn}=-r_n,\quad r_n=\sum_{m\ne n}r_{nm}<\infty.
\]
The matrices $Q^0$ and $R$ are irreducible.
The switching state $\Lambda_t=(E_t,N_t)\in\mathcal E\times\mathcal I$
records the active edge and noise level. Its off-diagonal rates at
spatial position $x\in\R^N$ are
\[
q_{(e,n),(f,m)}(x)=
\begin{cases}
q_{ef}(x),&f\ne e,\ m=n,\\
r_{nm},&f=e,\ m\ne n,\\
0,&f\ne e,\ m\ne n.
\end{cases}
\]
For one-dimensional Brownian motion $W$, the vector
$X_t=(X_t^1,\ldots,X_t^N)^\top\in\R^N$ of node states solves
\begin{equation}\label{eq:local-network}
dX_t=(A_{E_{t-}}X_t+c_{E_{t-}})\,dt
     +[g_{E_{t-}}(X_t)+s_{N_{t-}}]\mathbf v_1\circ dW_t.
\end{equation}
For $e=\{j,k\}$, the coupling term is
$-\kappa_eL_ex=-\kappa_e(x_j-x_k)(\mathbf v_j-\mathbf v_k)$:
this term couples only the two endpoints of the active edge. The term $-ax_1\mathbf v_1$
damps node~1, $\delta x$ gives uniform growth, and $c_e$ is an affine
input. Noise acts only at node~1, with state-dependent amplitude $g_e$
and level-dependent amplitude $s_n$.
This is a continuous-time stochastic network built from the pairwise
averaging mechanism of \citet[Section~I-A, (1)--(2)]{BoydEtAl2006}
and the switching pinning framework of
\citet[Section~2, (3)]{HanEtAl2014}, with persistent Brownian forcing
and a countable noise-level coordinate.

The degeneracy of every frozen regime is essential to this example.
For $e=\{j,k\}$, choose
$0\ne z\in\operatorname{span}\{\mathbf v_1,\mathbf v_j-\mathbf v_k\}^{\perp}$,
which is possible because $N\ge3$. With $(E_t,N_t)\equiv(e,n)$,
$z^\top A_e=\delta z^\top$ and
$z^\top\mathbf v_1=0$, so
\begin{equation}\label{eq:frozen-projection}
	z^\top X_t=e^{\delta t}z^\top x+
	\left(\int_0^te^{\delta(t-s)}\,ds\right)z^\top c_e.
\end{equation}
Thus each frozen transition law is supported on an affine hyperplane,
and different initial projections have disjoint supports. In particular,
no frozen regime admits a common nonzero minorizing measure on a
spatial ball. 

To construct a Lyapunov function for the joint process, use the
coupled Lyapunov equations for the reference drift $D_e$ and rates $Q^0$:
\[
D_e^\top P_e+P_eD_e+\sum_{f\ne e}q^0_{ef}(P_f-P_e)=-I_N,
\qquad e\in\mathcal E.
\]
Lemma~\ref{lem:network-lyapunov} gives their unique symmetric solution,
with $P_e$ positive definite. For the reference switching dynamics,
these equations make the generator applied to $x^\top P_ex$ equal to
$-|x|^2$. The following assumptions control the added growth, the
change in edge rates, and the noise levels.
\begin{enumerate}[label=\textup{(N\arabic*)},ref=N\arabic*,leftmargin=*]
\item\label{cond:network-edge}
The functions $q_{ef}$, $e\ne f$, are bounded and Lipschitz.
The directed graph with edge $e\to f$ whenever
$q_{ef}(x)>0$ for every $x\in\R^N$ is strongly connected.
\item\label{cond:network-margin}
The matrices $P_e$ satisfy the margin condition
\begin{equation}\label{eq:network-margin}
\begin{aligned}
p_{\max}&=\max_e\|P_e\|,\qquad
\varepsilon_Q=
\max_e\sum_{f\ne e}\sup_x|q_{ef}(x)-q^0_{ef}|\,\|P_f-P_e\|,\\
\eta_0&=1-2\delta p_{\max}-\varepsilon_Q>0.
\end{aligned}
\end{equation}
\item\label{cond:network-level}
There is a function $\chi:\mathcal I\to[1,\infty)$ with finite
sublevel sets and constants $\kappa,C>0$, $B\ge0$ such that
\begin{equation}\label{eq:network-level-condition}
\begin{aligned}
&\sum_{m\ne n}r_{nm}|\chi(m)-\chi(n)|<\infty,\\
&R\chi(n):=\sum_{m\ne n}r_{nm}(\chi(m)-\chi(n))
   \le-\kappa\chi(n)+B,\qquad
1+s_n^2\le C\chi(n),\quad n\in\mathcal I.
\end{aligned}
\end{equation}
\end{enumerate}
For $\theta>0$, define
$V(x,e,n)=1+x^\top P_ex+\theta\chi(n)$ on
$\R^N\times\mathcal E\times\mathcal I$.

\Needspace{7\baselineskip}
\begin{prop}[Ergodicity of the switching network]
\label{prop:countable-network}
Under conditions \ref{cond:network-edge}--\ref{cond:network-level},
the joint process $(X_t,E_t,N_t)$ defined by \eqref{eq:local-network}
and the above switching rates has a unique nonexplosive strong Markov
solution. For all sufficiently large $\theta$, its transition semigroup
satisfies \eqref{eq:weighted-ergodicity} with weight $V$, and its
unique invariant probability $\pi$ satisfies $\pi(V)<\infty$.
\end{prop}
\begin{proof}
The It\^o coefficients are
\[
\begin{aligned}
b_{e,n}(x)&=A_ex+r_{e,n}(x),\qquad
r_{e,n}(x)=c_e+\tfrac12[g_e(x)+s_n]\partial_1g_e(x)\mathbf v_1,\\
\sigma_{e,n}(x)&=[g_e(x)+s_n]\mathbf v_1.
\end{aligned}
\]
For each $(e,n)$, boundedness of $g_e$ and its first two
derivatives gives global Lipschitz continuity. The total jump rate
$\sum_{f\ne e}q_{ef}(x)+r_n$ is bounded on each finite regime set.
Thus the rates satisfy Assumption~\ref{ass:local-construction}(ii).

Fix $n_*\in\mathcal I$ and take
$\mathcal R=\mathcal J=\mathcal E\times\{n_*\}$.
The Stratonovich drift is $b_{e,n}^\circ(x)=A_ex+c_e$, and the coefficient
representation holds on $\mathcal J$ with
$B_{e,n}=\mathbf v_1$, $f_{e,n}=0$, and $R_{e,n}(x)=[g_e(x)+s_n]^{-1}$.
Let $\mathcal W$ be the smallest subspace containing $\mathbf v_1$ and invariant
under every $A_e$. If $\mathbf v_j\in\mathcal W$ and $e=\{j,k\}$, then
\[
\mathbf v_k=\mathbf v_j+\frac1{\kappa_e}
\left(A_e\mathbf v_j-\delta \mathbf v_j+a\one_{\{j=1\}}\mathbf v_1\right)\in\mathcal W.
\]
Connectivity of the spatial graph gives $\mathcal W=\R^N$, hence
$\mathcal V_N=\R^N$ in \eqref{eq:rank-iteration}.
The joint regime graph and the induced graph on $\mathcal J$
are strongly connected, verifying Assumption~\ref{ass:smoothing}(i).
The stated smoothness gives Assumption~\ref{ass:local-construction}(i), completing the verification
of Assumptions~\ref{ass:local-construction},
\ref{ass:coefficient-representation}, and~\ref{ass:smoothing}.

Level transitions leave $P_e$ unchanged. The matrix equation therefore gives
\[
\begin{aligned}
M_e(x)&:=A_e^\top P_e+P_eA_e+\sum_{f\ne e}q_{ef}(x)(P_f-P_e)\\
&=-I+2\delta P_e+
   \sum_{f\ne e}(q_{ef}(x)-q^0_{ef})(P_f-P_e)
\preceq-\eta_0 I.
\end{aligned}
\]
Boundedness of $g_e$, $\partial_1g_e$, and $c_e$, together with
Young's inequality, yields a constant $C_0<\infty$, independent of
$x,e,n$, such that
\[
2x^\top P_er_{e,n}(x)+[g_e(x)+s_n]^2\mathbf v_1^\top P_e\mathbf v_1
\le\frac{\eta_0}2|x|^2+C_0(1+s_n^2).
\]
Consequently
\[
\begin{aligned}
\A V(x,e,n)
&=x^\top M_e(x)x+2x^\top P_er_{e,n}(x)
  +[g_e(x)+s_n]^2\mathbf v_1^\top P_e\mathbf v_1+\theta R\chi(n)\\
&\le-\frac{\eta_0}2|x|^2-(\theta\kappa-C_0C)\chi(n)+\theta B.
\end{aligned}
\]
Choose $\theta>C_0C/\kappa$ and put
\[
a_V=\min\left\{\frac{\eta_0}{2p_{\max}},
                    \kappa-\frac{C_0C}{\theta}\right\}>0.
\]
Then $\A V\le-a_VV+\theta B+a_V$.
Moreover, with $\underline p=\min_e\lambda_{\min}(P_e)>0$,
\[
\begin{aligned}
V(x,e,n)&\ge1+\underline p|x|^2+\theta\chi(n),\\
J_V(x,e,n)&\le |x|^2\sum_{f\ne e}q_{ef}(x)\|P_f-P_e\|
   +\theta\sum_{m\ne n}r_{nm}|\chi(m)-\chi(n)|.
\end{aligned}
\]
These bounds give compact sublevel sets for $V$ and local
boundedness of $J_V$, verifying Assumption~\ref{ass:joint-lyapunov}.
Corollary~\ref{thm:coefficient-ergodicity} gives nonexplosion and
weighted exponential ergodicity of $(X_t,E_t,N_t)$.
\end{proof}

\begin{remark*}[An explicit countable network]
Take $N=3$, $\mathcal E=\{e_1,e_2\}$ with
$e_1=\{1,2\}$ and $e_2=\{2,3\}$, and set
$a=\kappa_{e_1}=\kappa_{e_2}=1$, $c_{e_1}=c_{e_2}=0$,
$\delta=1/100$, and $q^0_{e_1e_2}=q^0_{e_2e_1}=1$.
Choose
\[
\begin{gathered}
g_{e_1}(x)=g_{e_2}(x)=2+\sin x_1,\qquad
q_{e_1e_2}(x)=1+\tfrac1{100}\sin x_1,\qquad
q_{e_2e_1}(x)=1+\tfrac1{100}\cos x_1,\\
\mathcal I=\Nzero,\qquad s_n=\sqrt n,\qquad
r_{n,n+1}=1\ (n\ge0),\qquad r_{n,n-1}=n\ (n\ge1),
\end{gathered}
\]
with all other off-diagonal level rates zero.
The edge rates are bounded and Lipschitz and are at least $99/100$,
which verifies \ref{cond:network-edge}. The coupled Lyapunov matrices are
\[
P_{e_1}=\frac1{697}
\begin{pmatrix}386&436&460\\436&1137&1153\\460&1153&3240\end{pmatrix},
\qquad
P_{e_2}=\frac1{697}
\begin{pmatrix}361&221&227\\221&1842&1846\\227&1846&2543\end{pmatrix}.
\]
Positivity gives $\|P_e\|\le\operatorname{tr}P_e<7$; the row sums
of the symmetric difference give
$\|P_{e_2}-P_{e_1}\|\le1623/697<3$. Thus
$p_{\max}<7$, $\varepsilon_Q<3/100$, and $\eta_0>83/100$,
verifying \ref{cond:network-margin}.
For $\chi(n)=1+n$,
\[
\sum_{m\ne n}r_{nm}|\chi(m)-\chi(n)|=1+n,\qquad
R\chi(n)=-\chi(n)+2,\qquad 1+s_n^2=\chi(n),
\]
which verifies \ref{cond:network-level} with $\kappa=C=1$ and $B=2$.
Proposition~\ref{prop:countable-network} therefore gives exponential
ergodicity with state-dependent rates, multiplicative noise, and
unbounded noise amplitudes. In every frozen regime,
\eqref{eq:frozen-projection} still gives an unforced projection
expanding at rate $1/100$.
\end{remark*}

Table~\ref{tab:network-comparison} compares the hypotheses of
related results. The network violates the frozen regularity and
small-set hypotheses in the first four rows, by rank-one noise and
\eqref{eq:frozen-projection}.

\begin{table}[!htbp]
\caption{Comparison with related criteria.}\label{tab:network-comparison}
\small
\centering
\setlength{\tabcolsep}{5pt}
\renewcommand{\arraystretch}{1.05}
\begin{tabular}{@{}>{\raggedright\arraybackslash}p{0.44\linewidth}
>{\raggedright\arraybackslash}p{\dimexpr0.56\linewidth-2\tabcolsep\relax}@{}}
\toprule
Reference & Required condition \\
\midrule
\citet[Ass.~3.2(i), Thm.~3.2]{Kunwai2022}
& Uniform ellipticity in every regime. \\[2pt]
\citet[Ass.~4.1, Thm.~6.3]{XiZhu2017}
& Strong Feller frozen diffusions with Lebesgue densities. \\[2pt]
\citet[Def.~3.3, Thm.~3.4]{TongMajda2016}
& Minorization for an accessible frozen regime. \\[2pt]
\citet[Thms.~1.7--1.8]{CloezHairer2015}
& Small Lyapunov sublevel sets for one frozen semigroup. \\[2pt]
\citet[(A4), (A6), Thms.~3.5, 4.5]{LiXi2022TV}
& Scalar radial drift yielding a nonsingular $M$-matrix or a
negative stationary average. \\[2pt]
\citet[(14), Thm.~4.6]{HuNualartSunXie2019}
& Uniform H\"ormander condition in every regime. \\
\bottomrule
\end{tabular}
\par\smallskip
\begin{minipage}{\linewidth}
\footnotesize
For the finite-level analogue with constant rates and $\delta>0$,
the Li--Xi criteria fail because
every radial coefficient is at least $2\delta$; their joint minorization
assumption remains compatible. For the density theorem of Hu et~al., even when $g_e\in C^\infty$,
the covector $z^\top$ in \eqref{eq:frozen-projection} annihilates
all noise-generated brackets in the frozen regime.
\end{minipage}
\end{table}

The no-switch event has positive probability, so the full
transition law retains a singular component at every finite time.

\appendix
\Needspace{24\baselineskip}
\section{Local analytic ingredients}\label{app:analytic}
\subsection{Deterministic endpoint estimates}\label{app:deterministic}
We prove the endpoint differentiation and inverse-map estimates used
in Section~\ref{sec:rank-cycle}. For the endpoint result, fix
$x\in\R^d$, $T>0$, and a prescribed switching schedule $\lambda$
on $[0,T]$ as defined in Section~\ref{subsec:controlled-paths}, and set
$\mathcal I=\{\lambda_t:t\in[0,T]\}$.
Recall the controlled equation and endpoint map from
\eqref{eq:skeleton} in Section~\ref{subsec:controlled-paths}.
Let $u\in L^2([0,T];\R^m)$ have a controlled path $\phi$ on
$[0,T]$, and let $O\subset\R^d$ be an open neighborhood of the
compact set $\phi([0,T])$.
Write $\sigma_i^{(k)}$ for the $k$th column of $\sigma_i$ and
$(u_t)_k$ for the $k$th component of $u_t$.

\begin{lem}[Differentiating the endpoint]
\label{lem:endpoint-directions}
Suppose $b_i^\circ,\sigma_i$ are $C^2$ on $O$ for every
$i\in\mathcal I$. Then $\mathcal E_{x,T,\lambda}$ is continuously
Fr\'echet differentiable on an $L^2$ neighborhood of $u$, with
\[
D_u\mathcal E_{x,T,\lambda}(u)v
 =\int_0^T J_{T,s}\sigma_{\lambda_s}(\phi_s)v_s\,ds,
\qquad v\in L^2([0,T];\R^m).
\]
Here $\partial_tJ_{t,s}=C_tJ_{t,s}$, $J_{s,s}=I$, and
\[
C_s=Db_{\lambda_s}^\circ(\phi_s)
 +\sum_{k=1}^m (u_s)_kD\sigma_{\lambda_s}^{(k)}(\phi_s).
\]
If the derivative is onto $\R^d$, its restriction to the span of
$d$ orthonormal step functions is invertible.
\end{lem}

\begin{proof}[Proof of Lemma~\ref{lem:endpoint-directions}]
Put $K=\phi([0,T])$ and choose a bounded open set $V$ and a cutoff
$\chi$ with
\[
K\subset V,\qquad \overline V\subset O,\qquad
\chi\in C_c^\infty(O),\qquad \chi|_V=1.
\]
The zero extensions of
$\widetilde b_i^\circ=\chi b_i^\circ$ and
$\widetilde\sigma_i=\chi\sigma_i$ are globally Lipschitz and $C^2$,
with bounded derivatives through order two.
For $w\in L^2([0,T];\R^m)$, define
$\widehat w\in L^2([0,T];\R^{(m+1)|\mathcal I|})$ by
\[
\widehat w_{i,0}(t)=\one_{\{\lambda_t=i\}},\qquad
\widehat w_{i,k}(t)=\one_{\{\lambda_t=i\}}(w_t)_k,
\quad i\in\mathcal I,\quad 1\le k\le m.
\]
The localized path $\widetilde\phi^w$, started from $x$, solves
\[
\dot{\widetilde\phi}_t^w=\sum_{i\in\mathcal I}
\left[\widetilde b_i^\circ(\widetilde\phi_t^w)\widehat w_{i,0}(t)
+\sum_{k=1}^m\widetilde\sigma_i^{(k)}(\widetilde\phi_t^w)
\widehat w_{i,k}(t)\right].
\]
Apply \citet[Propositions~2.3--2.4, 2.7, and~2.11]{Scagliotti2023}
after rescaling $t=Ts$ and the control to $T\widehat w(Ts)$.
The locally Lipschitz G\^ateaux differential in operator norm is a
continuous Fr\'echet differential, also after composition with the
bounded affine map $w\mapsto\widehat w$.
Since $\chi=1$ near $K$, uniqueness gives $\widetilde\phi^u=\phi$.
For $z_t=\left.\partial_\varepsilon\widetilde\phi_t^{u+\varepsilon v}
\right|_{\varepsilon=0}$, differentiation gives
\[
\dot z_t=C_tz_t+\sigma_{\lambda_t}(\phi_t)v_t,\qquad z_0=0.
\]
The bound $\|C_t\|\le c(1+|u_t|)$, with $c<\infty$, makes $C_t$ integrable, so
variation of constants gives the derivative formula.
Proposition~2.3 gives, for sufficiently small $\|h\|_{L^2}$,
\[
\|\widetilde\phi^{u+h}-\phi\|_\infty
\le C\|h\|_{L^2}<\operatorname{dist}(K,V^c).
\]
The paths stay in $V$, where uniqueness identifies the localized
and original endpoint maps.

If $L=D_u\mathcal E_{x,T,\lambda}(u)$ is onto, choose $v_1,\ldots,v_d\in L^2$
with independent images $Lv_j$. On
$I_{n,r}=[rT2^{-n},(r+1)T2^{-n})$, including $T$ in the final
interval, define the dyadic step approximations
\[
v_j^{(n)}(t)=\sum_{r=0}^{2^n-1}\one_{I_{n,r}}(t)
\frac{2^n}{T}\int_{I_{n,r}}v_j(s)\,ds.
\]
Since $\|v_j^{(n)}-v_j\|_{L^2}\to0$ and $L$ is bounded,
\[
\det[Lv_1^{(n)}\ \cdots\ Lv_d^{(n)}]
\longrightarrow\det[Lv_1\ \cdots\ Lv_d]\ne0.
\]
For large $n$, Gram--Schmidt gives the required orthonormal step
functions without changing their span.
\end{proof}

The next lemma gives an image ball and an upper Jacobian bound
for a map close to an invertible linear map. Fix $\xi_0\in\R^d$,
$\rho>0$, and
\[
I=B(\xi_0,\rho)=\{\xi\in\R^d:|\xi-\xi_0|<\rho\},\qquad
\overline I=\{\xi\in\R^d:|\xi-\xi_0|\le\rho\}.
\]
Let $A\in\R^{d\times d}$ be invertible, and fix
$0<s\le\|A^{-1}\|^{-1}$ and $0\le\kappa<1$.

\begin{lem}[Quantitative inverse function theorem]
\label{lem:uniform-finite-dimensional-inverse}
Let $U\subset\R^d$ be open with $\overline I\subset U$, and let
$F\in C^1(U;\R^d)$. Assume
\[
\sup_{\xi\in\overline I}\|A^{-1}(DF(\xi)-A)\|\le\kappa.
\]
Then $F|_I$ is a $C^1$ diffeomorphism onto its image, and
\begin{equation}\label{eq:quantitative-inverse}
\begin{gathered}
|F(\xi)-F(\eta)|\ge s(1-\kappa)|\xi-\eta|,
\qquad \xi,\eta\in\overline I,\\
B(F(\xi_0),s(1-\kappa)\rho)\subset F(I),\qquad
|F^{-1}(z)-\xi_0|\le\frac{|z-F(\xi_0)|}{s(1-\kappa)},\quad z\in F(I),\\
\sup_{\xi\in I}|\det DF(\xi)|
\le\bigl(\sup_{\xi\in I}\|DF(\xi)\|\bigr)^d
\le((1+\kappa)\|A\|)^d.
\end{gathered}
\end{equation}
In particular, if $\kappa\le1/2$ and $|F(\xi_0)-y|<s\rho/4$,
then $B(y,s\rho/4)\subset F(I)$.
\end{lem}

\begin{proof}
For $\psi(\xi)=A^{-1}(F(\xi)-F(\xi_0))$, we have
$\psi(\xi_0)=0$ and $\|D\psi-I\|\le\kappa$.
Apply \citet[Appendix~C, Theorem~C.1 and estimate~(C.4)]{Salamon2016Measure}
with any contraction bound $\alpha\in(\kappa,1)$ and let
$\alpha\downarrow\kappa$. This gives a $C^1$ inverse and
\[
B(0,(1-\kappa)\rho)\subset\psi(I),\qquad
|\psi(\xi)-\psi(\eta)|\ge(1-\kappa)|\xi-\eta|.
\]
Continuity extends the second bound to $\overline I$; $|Az|\ge s|z|$
gives the first two lines of \eqref{eq:quantitative-inverse}.
Hadamard's inequality gives
\[
|\det DF|\le\|DF\|^d
\le\bigl(\|A\|(1+\|A^{-1}(DF-A)\|)\bigr)^d
\le((1+\kappa)\|A\|)^d.
\]
For $\kappa\le1/2$, the center condition gives
$B(y,s\rho/4)\subset B(F(\xi_0),s\rho/2)\subset F(I)$.
\end{proof}

\subsection{Cameron--Martin averaging}\label{app:averaging}
We state the averaging criterion independently of any SDE. It combines
finite-dimensional Wiener shifts with change of variables. Related
image-measure lower bounds appear in
\citet[Proposition~3.2]{Shirikyan2017}.

Fix $T>0$. Use the Wiener space, Cameron--Martin space $\mathcal H$, directions
$e_j$, and shifts $h_\xi$ defined in
Section~\ref{subsec:local-minorization}. Let $\alpha$ index a Borel
subset of a Euclidean space, and let $Q\subset\R^d$ be a closed ball
of positive radius with an open neighborhood $U_Q\supset Q\cup\{0\}$.
Consider a path-valued family
$(\alpha,\xi,w)\mapsto Y^\alpha(\xi;w)\in C([0,T];\R^d)$,
$\xi\in U_Q$, and prescribed path sets
$\mathcal T_\alpha\subset C([0,T];\R^d)$. Define
\[
\begin{aligned}
G_{\alpha,w}(\xi)&=Y_T^\alpha(\xi;w),\\
\Pi_\alpha^{\rm tube}(D)
&=\PP\{Y_T^\alpha(0;W)\in D,\,
Y^\alpha(0;W)\in\mathcal T_\alpha\}.
\end{aligned}
\]
For Theorem~\ref{prop:local}, this is the solution family of
\eqref{eq:shifted-sde}, with $\alpha=(x,\tau)$.
Choose a bounded open ball $I$ with $\overline I\subset Q$, an open
ball $O\subset\R^d$, and constants $q>0$ and $0<J<\infty$.

\Needspace{6\baselineskip}
\begin{lem}[Cameron--Martin averaging]\label{lem:averaging}
Assume the following conditions.
\begin{enumerate}[label=\textup{(\roman*)},leftmargin=*]
\item The derivatives $\dot e_j$ are step functions and the paths are
orthonormal in $\mathcal H$:
\[
\langle e_j,e_k\rangle_{\mathcal H}
=\int_0^T\dot e_j(t)\cdot\dot e_k(t)\,dt=\delta_{jk},
\qquad 1\le j,k\le d.
\]
\item The maps $(\alpha,\xi,w)\mapsto Y^\alpha(\xi;w)$ and
$(\alpha,y)\mapsto\one_{\mathcal T_\alpha}(y)$ are jointly Borel.
For each fixed $\alpha$ and $\xi\in Q$,
\[
Y^\alpha(\xi;W)=Y^\alpha(0;W+h_\xi)
\quad\PP\text{-almost surely as paths}.
\]
\item For each $\alpha$, there is an event $E_\alpha$ with
$\PP(E_\alpha)\ge q$ on which every $Y^\alpha(\xi;W)$, $\xi\in I$,
belongs to $\mathcal T_\alpha$, and $G_{\alpha,W}|_I$ is a $C^1$
diffeomorphism onto its image satisfying
\[
O\subset G_{\alpha,W}(I),\qquad
\sup_{\xi\in I}|\det D_\xi G_{\alpha,W}(\xi)|\le J.
\]
\end{enumerate}
Then there is a common $c>0$, depending only on $q,I,J$, such that
\[
\Pi_\alpha^{\rm tube}(D)\ge c\Leb(D\cap O),
\qquad D\in\B(\R^d).
\]
\end{lem}

\begin{proof}
The finite-dimensional shift method appears in
\citet[Lemma~4.5 and Appendix~A]{HerzogMattingly2015}.
By (i), refining the finitely many step-function partitions gives
$0=t_0<\cdots<t_N=T$, with $\dot e_j=a_{j,r}\in\R^m$ almost everywhere
on $[t_{r-1},t_r)$. Define the continuous linear functionals
\[
Z_j(w)=\sum_{r=1}^N a_{j,r}\cdot(w(t_r)-w(t_{r-1})),\qquad
Z_j(W)=\int_0^T\dot e_j(t)\cdot dW_t.
\]
The vector $Z(W)=(Z_1(W),\ldots,Z_d(W))$ is centered Gaussian with
covariance $(\langle e_j,e_k\rangle_{\mathcal H})=I_d$, hence standard
Gaussian. Moreover, for every $w$ and $\xi$,
\[
Z(w+h_\xi)=Z(w)+\xi.
\]

Fix $\alpha$ and $D\in\B(\R^d)$.
Choose a Borel representative of $E_\alpha$ on which (iii) holds. Set $M=\sqrt{4d/q}$ and
put $E'_\alpha=\{w\in E_\alpha:|Z(w)|\le M\}$.
Markov's inequality gives $\PP(E'_\alpha)\ge q/2$.
For $w\in E'_\alpha$, write $G=G_{\alpha,w}$. Change of variables
under (iii) gives
\[
\begin{aligned}
\int_I\one_D(G(\xi))\,d\xi
&=\int_{G(I)}
\frac{\one_D(y)}{|\det D_\xi G(G^{-1}(y))|}\,dy\\
&\ge J^{-1}\Leb(D\cap O).
\end{aligned}
\]
Integrating this bound by Tonelli's theorem, using (ii), gives
\[
\begin{aligned}
K_\alpha(D)
&:=\int_I\E[\one_{E'_\alpha}(W)
             \one_D(G_{\alpha,W}(\xi))]\,d\xi\\
&\ge\frac{\PP(E'_\alpha)}{J}\Leb(D\cap O)
 \ge\frac{q}{2J}\Leb(D\cap O).
\end{aligned}
\]
By (ii), the function
\[
f_\alpha(w)=\one_D(Y_T^\alpha(0;w))
\one_{\mathcal T_\alpha}(Y^\alpha(0;w))
\]
is Borel and satisfies $\E f_\alpha(W)=\Pi_\alpha^{\rm tube}(D)$.
For each fixed $\xi$, (ii) identifies the shifted paths almost surely,
and (iii) supplies the path restriction on $E'_\alpha$. Tonelli's
theorem integrates this jointly Borel equality over $I$, giving
\[
K_\alpha(D)=\int_I\E[\one_{E'_\alpha}(W)
f_\alpha(W+h_\xi)]\,d\xi.
\]

For fixed $\xi$, (i) gives $\|h_\xi\|_{\mathcal H}^2=|\xi|^2$ and
$\int_0^T\dot h_\xi\cdot dW=\xi\cdot Z(W)$. Apply the Cameron--Martin
formula in \citet[Theorem~3.41 and (3.13)]{Hairer2009SPDE}
to $w\mapsto\one_{E'_\alpha}(w-h_\xi)f_\alpha(w)$. With
$R_I=\sup_{\xi\in I}|\xi|$, it gives
\[
\begin{aligned}
\E[\one_{E'_\alpha}(W)f_\alpha(W+h_\xi)]
&=\E[\one_{E'_\alpha}(W-h_\xi)f_\alpha(W)
e^{\xi\cdot Z(W)-|\xi|^2/2}]\\
&\le e^{R_IM+R_I^2/2}\Pi_\alpha^{\rm tube}(D).
\end{aligned}
\]
Here $Z(W-h_\xi)=Z(W)-\xi$ gives $|Z(W)-\xi|\le M$ on the
indicated event, and
$\xi\cdot Z(W)-|\xi|^2/2
=\xi\cdot(Z(W)-\xi)+|\xi|^2/2\le R_IM+R_I^2/2$.
Integrating the resulting bound over $I$ yields
\[
K_\alpha(D)\le\Leb(I)e^{R_IM+R_I^2/2}
\Pi_\alpha^{\rm tube}(D).
\]
Combining the bounds for $K_\alpha(D)$ proves the claim with
\[
c=\frac{q}{2J\Leb(I)e^{R_IM+R_I^2/2}}>0.
\]
\end{proof}

\subsection{Support of the shifted family}\label{app:family}
The following lemma supplies the family support and parameter
continuity used in Theorem~\ref{prop:local}. Both paths and their
control-parameter derivatives are approximated on one event of
positive probability.
Fix $T>0$ and a regime sequence
$\boldsymbol\iota=(\iota_0,\ldots,\iota_n)\in\Sset^{n+1}$.
Use the spaces $\mathcal H,C_0$, the directions $e_j$, and the
sets $Q,U_Q$ from Section~\ref{subsec:local-minorization}, together
with $a,h_\xi$ and the norm \eqref{eq:family-norm}.
For $\alpha=(x,\tau)\in\R^d\times\Delta_n(T)$, the initial point is
$x$ and the prescribed switching schedule is $\lambda^{\boldsymbol\iota,\tau}$.
For a sequence $\alpha_j=(x_j,\tau^{(j)})$, write
$\tau^{(j)}=(\tau_1^{(j)},\ldots,\tau_n^{(j)})$ for its $j$th
switching-time vector.

Fix $\alpha_0=(x_0,\tau^0)$ and $u_0\in L^2([0,T];\R^m)$.
The deterministic reference family solves
\[
\begin{aligned}
\dot\phi_t^0(\xi)
&=b_{\lambda_t^{\boldsymbol\iota,\tau^0}}^\circ(\phi_t^0(\xi))
 +\sigma_{\lambda_t^{\boldsymbol\iota,\tau^0}}(\phi_t^0(\xi))
   [u_0(t)+a(t)\xi],\\
\phi_0^0(\xi)&=x_0,\qquad \xi\in U_Q.
\end{aligned}
\]

\begin{lem}[Regularity and support of the shifted family]
\label{lem:classical-family}
Assume that $\dot e_1,\ldots,\dot e_d$ are step functions and,
for every regime $i$ in the fixed sequence $\boldsymbol\iota$,
\[
b_i\in C_c^{d+3}(\R^d;\R^d),\qquad
\sigma_i\in C_c^{d+4}(\R^d;\R^{d\times m}).
\]
Then the solutions of \eqref{eq:shifted-sde} have a version
with the following properties.
\begin{enumerate}[label=\textup{(\roman*)},leftmargin=*]
\item 
The map
$(\alpha,\xi,w)\longmapsto Y^\alpha(\xi;w)$
from
$(\R^d\times\Delta_n(T))\times U_Q\times C_0([0,T];\R^m)$
to $C([0,T];\R^d)$
is jointly Borel measurable. There is a Borel set $\Omega_*$ of
Wiener measure one such that, for every $w\in\Omega_*$ and every
$\alpha$, the family is $C^1$ in $\xi$ and both
$(t,\xi)\mapsto Y_t^\alpha(\xi;w)$ and
$(t,\xi)\mapsto D_\xi Y_t^\alpha(\xi;w)$ are continuous.
\item
For every $\alpha_0$, every $u_0\in L^2([0,T];\R^m)$, and every
$\varepsilon>0$,
\begin{equation}\label{eq:family-support}
\PP\{\|Y^{\alpha_0}-\phi^0\|_{Q,1}<\varepsilon\}>0.
\end{equation}
\item
Let $\alpha_j=(x_j,\tau^{(j)})$ and $\alpha_0=(x_0,\tau^0)$ belong
to $\R^d\times\Delta_n(T)$, with
$x_j\to x_0$ and $\tau^{(j)}\to\tau^0$.
For every $w\in\Omega_*$,
\begin{equation}\label{eq:schedule-stability}
\|Y^{\alpha_j}-Y^{\alpha_0}\|_{Q,1}\longrightarrow0.
\end{equation}
All these solutions use the same regime sequence $\boldsymbol\iota$.
\item 
For each fixed $\alpha$ and $\xi\in Q$,
$Y^\alpha(\xi;W)=Y^\alpha(0;W+h_\xi)$ almost surely as paths.
\end{enumerate}
\end{lem}

\begin{proof}
We obtain regularity and continuity from frozen flows, then prove
support and identify the shifted It\^o equation.

\emph{1. Construct the augmented frozen flows and their derivatives.}
Adjoin $\xi$ as a constant state coordinate so that spatial flow
derivatives also give derivatives in $\xi$.
Refining the finitely many step-function partitions gives
$0=t_0<\cdots<t_M=T$, with
$a(t)=a_\ell\in\R^{m\times d}$ almost everywhere on
$[t_{\ell-1},t_\ell)$.
Choose $\chi_Q\in C_c^\infty(\R^d)$ equal to one near
$\overline U_Q$.
For $z=(x,\xi)\in\R^{2d}$, each visited regime $i$, and
$1\le\ell\le M$, define
\begin{equation}\label{eq:augmented-fields}
\begin{aligned}
V_{i,\ell,0}(z)
 &=\chi_Q(\xi)\bigl(b_i^\circ(x)+\sigma_i(x)a_\ell\xi,0\bigr),\\
V_{i,\ell,k}(z)&=\chi_Q(\xi)\bigl(\sigma_i^{(k)}(x),0\bigr),
\qquad 1\le k\le m.
\end{aligned}
\end{equation}
Their It\^o drift is
$\widehat b_{i,\ell}=V_{i,\ell,0}
+\frac12\sum_kDV_{i,\ell,k}V_{i,\ell,k}$.
On $\R^d\times U_Q$, it equals
$(b_i(x)+\sigma_i(x)a_\ell\xi,0)$ by \eqref{eq:corrected}.
Define $\widehat\Psi^{i,\ell}_{s,t}(z)=\widehat Z_t$ by the
augmented It\^o equation
\[
d\widehat Z_t=\widehat b_{i,\ell}(\widehat Z_t)\,dt
 +\sum_{k=1}^m V_{i,\ell,k}(\widehat Z_t)\,dW_t^k,
\qquad \widehat Z_s=z.
\]
Its covariance characteristic is
$\widehat a_{i,\ell}(z,z')
=\sum_k V_{i,\ell,k}(z)V_{i,\ell,k}(z')^\top$.
Suppress $(i,\ell)$ in the following bounds.
The coefficient assumptions and the parameter cutoff give
\[
\begin{gathered}
\max_{|\nu|\le2}\sup_z\|D^\nu\widehat b(z)\|
 +\max_{|\nu|,|\mu|\le2}\sup_{z,z'}
    \|D_z^\nu D_{z'}^\mu\widehat a(z,z')\|<\infty,\\
\sum_k|V_k(z)-V_k(z')|^2\le c|z-z'|^2.
\end{gathered}
\]
The fields have compact support in both coordinates and are globally
Lipschitz and bounded; the displayed characteristic derivatives are
also Lipschitz because the next derivatives are bounded.
The SDE construction in
\citet[Theorem~2.4.3 and Corollary~2.4.5]{Kunita1986Flows}
gives a Brownian flow with these local characteristics.
For fixed $s,t$, its interval maps have representatives measurable
with respect to the Brownian increments on $[s,t]$.
Boundedness of the fields gives
$|\E[\widehat\Psi_{s,t}(z)-z]|+
\E|\widehat\Psi_{s,t}(z)-z|^2\le c_T(t-s)$.
Together with the displayed bounds, this verifies the hypotheses of
\citet[Theorem~1.4.1, with $k=2$]{Kunita1986Flows}, which gives
a flow of $C^1$ diffeomorphisms.

For each compact $K\subset\R^{2d}$ and integer $p\ge1$,
\citet[Lemmas~1.3.6 and~1.4.2]{Kunita1986Flows} give
\begin{equation}\label{eq:flow-moment-estimate}
\begin{aligned}
&\E|\widehat\Psi_{s,t}(z)-\widehat\Psi_{s',t'}(z')|^{2p}
 +\E\|D_z\widehat\Psi_{s,t}(z)
            -D_z\widehat\Psi_{s',t'}(z')\|^{2p}\\
&\qquad\le c_{K,p}
 \bigl(|z-z'|^{2p}+|s-s'|^p+|t-t'|^p\bigr),
\end{aligned}
\end{equation}
for $z,z'\in K$ and $0\le s\le t\le T$, $0\le s'\le t'\le T$.
For the derivative term, let the spatial difference-quotient
offsets in Lemma~1.4.2 tend to zero and apply Fatou's lemma.

The parameter $(s,t,z)$ has $2d+2$ real coordinates.
Extend the flows by the identity when $t<s$; the Lipschitz
projection $(s,t)\mapsto(s,\max\{s,t\})$ preserves the moment bound
on $[0,T]^2$. Taking $p>2d+2$, the Kolmogorov continuity theorem
in \citet[Theorem~1.3.7]{Kunita1986Flows}, applied to
\eqref{eq:flow-moment-estimate}, gives jointly continuous versions
of the flows and their derivatives.
Choosing Borel representatives on a countable dense parameter set
and using countably many compact boxes gives one Borel event
$\Omega_*$ of Wiener measure one, common to all $(i,\ell)$, on
which these continuity and flow properties hold simultaneously.
Set the flows to the identity off $\Omega_*$. On $\Omega_*$,
\[
\widehat\Psi^{i,\ell}_{s,s}(z)=z,\qquad
D_z\widehat\Psi^{i,\ell}_{s,s}(z)=I_{2d}.
\]

\emph{2. Recover the shifted family and prove \textup{(i)}, \textup{(iii)}.}
For $t>0$, refine the switching partition by the fixed mesh:
\[
\mathcal P(\tau,t)
=\{0,t\}\cup
 \bigl((\{\tau_1,\ldots,\tau_n\}\cup\{t_1,\ldots,t_M\})\cap(0,t)\bigr)
=\{0=s_0<\cdots<s_L=t\}.
\]
On $(s_{j-1},s_j)$ the pair
$(\lambda_r^{\boldsymbol\iota,\tau},a(r))=(i_j,a_{\ell_j})$ is constant
almost everywhere.
The frozen-flow representation of the shifted family is
\begin{equation}\label{eq:flow-composition}
\begin{aligned}
Y_t^\alpha(\xi)
&=\operatorname{pr}_x\mathcal F_t^\tau(x,\xi),\\
\mathcal F_t^\tau
&=\widehat\Psi^{i_L,\ell_L}_{s_{L-1},s_L}\circ\cdots\circ
 \widehat\Psi^{i_1,\ell_1}_{s_0,s_1},
\end{aligned}
\end{equation}
where $\operatorname{pr}_x(x,\xi)=x$ and $\mathcal F_0^\tau=\mathrm{id}$.
The second coordinate of each frozen flow stays equal to $\xi$,
and $\chi_Q(\xi)=1$ on $U_Q$; hence this representation solves
\eqref{eq:shifted-sde}. Its derivative is
\[
D_\xi Y_t^\alpha(\xi)
=\operatorname{pr}_xD_z\mathcal F_t^\tau(x,\xi)
 \begin{pmatrix}0_{d\times d}\\I_d\end{pmatrix}.
\]

Fix $w\in\Omega_*$. Step~1 gives $C^1$ frozen flows jointly continuous
in their interval endpoints. On each of the finitely many orderings
of $\mathcal P(\tau,t)$, the composition and its derivative are
continuous in $(t,x,\tau,\xi)$. When partition points coalesce,
\[
\sup_{z\in K}\bigl(
|\widehat\Psi^{i,\ell}_{u,v}(z)-z|
+\|D_z\widehat\Psi^{i,\ell}_{u,v}(z)-I_{2d}\|\bigr)\to 0
\quad\text{as }(u,v)\to(s,s)
\]
for every compact $K$. Deleting these identity factors gives the same
boundary values for the adjoining orderings. Thus the target and
its $\xi$ derivative are jointly continuous.
On a compact parameter neighborhood the intermediate images of the
finite compositions lie in compact sets. Uniform continuity gives
\[
\begin{aligned}
\|Y^{\alpha_j}-Y^{\alpha_0}\|_{Q,1}
&\le\sup_{\substack{t\le T\\\xi\in Q}}
 |\mathcal F_t^{\tau^{(j)}}(x_j,\xi)
             -\mathcal F_t^{\tau^0}(x_0,\xi)|\\
&\quad+\sup_{\substack{t\le T\\\xi\in Q}}
 \|D_z\mathcal F_t^{\tau^{(j)}}(x_j,\xi)
             -D_z\mathcal F_t^{\tau^0}(x_0,\xi)\|
 \to 0,
\end{aligned}
\]
which proves \eqref{eq:schedule-stability}.

For joint measurability, take measurable flow and derivative versions
on a countable dense set $\mathcal D_0$ of triples $(s,t,z)$.
Choose Borel approximations $\gamma_N(s,t,z)\in\mathcal D_0$ converging to
$(s,t,z)$, and define
\[
(\widehat\Psi,D_z\widehat\Psi)(s,t,z;w)=
\begin{cases}
\displaystyle\lim_{N\to\infty}
 (\widehat\Psi,D_z\widehat\Psi)(\gamma_N(s,t,z);w),&w\in\Omega_*,\\
(z,I_{2d}),&w\notin\Omega_*.
\end{cases}
\]
These are limits of jointly Borel maps. Formula
\eqref{eq:flow-composition} therefore gives jointly Borel evaluations
$Y_t^\alpha(\xi;w)$. Since evaluations at rational times generate
the Borel $\sigma$-field of $C([0,T];\R^d)$, the path-valued map is
Borel. This proves \textup{(i)} and \textup{(iii)}.

\emph{3. Prove the support assertion \textup{(ii)}.}
Fix $\alpha_0=(x_0,\tau^0)$ and $u_0\in L^2([0,T];\R^m)$.
Choose $R>0$ so that $B=(-R,R)^{2d}$ contains the supports of the
fields in \eqref{eq:augmented-fields} and the compact family
$\{(\phi_t^0(\xi),\xi):t\in[0,T],\ \xi\in Q\}$ in its interior;
let $c_*>0$ be the distance of this family from $\partial B$.
On the compact flat torus $M=\R^{2d}/(2R\mathbb Z)^{2d}$, define
\[
\widetilde V_{i,\ell,k}\bigl(z\bmod(2R\mathbb Z)^{2d}\bigr)
=\sum_{p\in\mathbb Z^{2d}}V_{i,\ell,k}(z+2Rp),\qquad 0\le k\le m.
\]
The sum is locally finite, has the regularity of $V_{i,\ell,k}$,
and equals that field on $B$.
Write $H^r(TM)$ for vector fields whose periodic components have
square-integrable weak derivatives through order $r$.
With $s=d+2$, the coefficient assumptions give
\[
\begin{gathered}
\widetilde V_{i,\ell,0}\in H^{s+1}(TM),\qquad
\widetilde V_{i,\ell,k}\in H^{s+2}(TM)\ (1\le k\le m),\\
\sum_{k=1}^m\|\widetilde V_{i,\ell,k}\|_{H^{s+2}}^2<\infty.
\end{gathered}
\]
Thus \citet[Theorem~3, pp.~964--965]{Aida1990FlowSupport}, with all
later noise fields zero, applies on this compact manifold.
Since $s>\dim(M)/2+1$, Sobolev embedding gives, for each
fixed regime and interval on which $a$ is constant,
each $0<L\le T$, every smooth control $u$, and every $\delta>0$,
\[
\begin{aligned}
\PP\left\{\sup_{t\le L}
d_{C^1}(\widetilde\Psi_t,\widetilde\psi_t^u)<\delta\right\}&>0,\\
d_{C^1}(f,g)&=\sup_{z\in M}d_M(f(z),g(z))
+\sup_{z\in M}\|D_zf-D_zg\|,
\end{aligned}
\]
where $\widetilde\Psi$ and $\widetilde\psi^u$ are its stochastic and
controlled flows, $d_M$ is the flat-torus distance, and derivative
matrices use the trivialization $TM=M\times\R^{2d}$.
For $\|u\|_{L^1}+\|v\|_{L^1}\le R_0$, subtracting the controlled
and variational equations and applying Gronwall's inequality gives
\[
\sup_{t\le L}d_{C^1}(\widetilde\psi_t^u,\widetilde\psi_t^v)
\le C_{R_0}\|u-v\|_{L^1}
\le C_{R_0}\sqrt L\,\|u-v\|_{L^2}.
\]
The constant depends on $R_0,T$ and the fields' bounded $C^2$ norms.
Smooth approximation in $L^2$ therefore gives the
same positive-probability assertion for every $L^2$ control.

Use Step~2's partition
$\mathcal P(\tau^0,T)=\{0=s_0<\cdots<s_{L_0}=T\}$.
Let $\widetilde\Psi_j(t)$ and $\widetilde\psi_j(t)$ be the frozen
stochastic and controlled torus flows, started from the identity at
$s_{j-1}$, with the corresponding fields and the restriction of $u_0$.
The Borel events
\[
E_j(\delta)=\left\{\sup_{s_{j-1}\le t\le s_j}
d_{C^1}(\widetilde\Psi_j(t),\widetilde\psi_j(t))<\delta\right\}
\]
have probabilities $p_j(\delta)>0$ and depend on Brownian increments
on disjoint intervals. Compose the factors as in Step~2, including
the partially completed last interval, to obtain $\widetilde F_t$
and $\widetilde f_t$. The controlled composition has coordinates
$(\phi_t^0(\xi),\xi)$ at $(x_0,\xi)$, $\xi\in Q$.
The chain rule, the controlled factors' bounded $C^2$ norms, and
the stochastic factors' bounded first derivatives on $E_j(\delta)$,
$\delta\le1$, give
\[
\begin{aligned}
\sup_{t\le T}d_{C^1}(\widetilde F_t,\widetilde f_t)
&\le C_0\sum_{j=1}^{L_0}\sup_{s_{j-1}\le t\le s_j}
d_{C^1}(\widetilde\Psi_j(t),\widetilde\psi_j(t))\\
&\le C\delta
\quad\text{on }\bigcap_jE_j(\delta).
\end{aligned}
\]
Choose $\delta\le1$ with $C\delta<\min\{\varepsilon/2,c_*/2\}$.
The torus paths from $(x_0,\xi)$, $\xi\in Q$, then remain in the
chart $B$ and their second coordinate remains $\xi$.
There the periodic fields equal $V_{i,\ell,k}$, and $\chi_Q=1$ near $Q$.
Pathwise uniqueness identifies their first coordinates on a countable
dense set of $\xi$; Step~2's $C^1$ continuity extends this equality
to $Q$ outside one null set. Their parameter derivatives consequently
agree on $\operatorname{int}Q$, and on its boundary by continuity.
With the same first-coordinate projection as in Step~2, we obtain
\[
\|Y^{\alpha_0}-\phi^0\|_{Q,1}
\le2\sup_{t\le T}d_{C^1}(\widetilde F_t,\widetilde f_t)
\le2C\delta<\varepsilon
\quad\text{on }\bigcap_jE_j(\delta).
\]
Independent increments now give
\[
\PP\{\|Y^{\alpha_0}-\phi^0\|_{Q,1}<\varepsilon\}
\ge\PP\left(\bigcap_{j=1}^{L_0}E_j(\delta)\right)
=\prod_{j=1}^{L_0}p_j(\delta)>0.
\]
This proves \textup{(ii)}.

\emph{4. Identify the Cameron--Martin shifts and prove \textup{(iv)}.}
Fix $\alpha=(x,\tau)$ and $\xi\in Q$, and let
$T_{h_\xi}w=w+h_\xi$, where $\dot h_\xi(t)=a(t)\xi$.
By \citet[Theorem~3.41 and (3.13)]{Hairer2009SPDE},
\[
\frac{d(\PP\circ T_{h_\xi}^{-1})}{d\PP}(W)
=\exp\left\{\int_0^T\dot h_\xi(r)\cdot dW_r
                  -\tfrac12\|h_\xi\|_{\mathcal H}^2\right\}>0.
\]
Thus translation preserves null sets and transfers uniform convergence
in probability on $[0,T]$ by absolute continuity.
Take a nonanticipating Borel strong-solution representative of
$Y^\alpha(0)$; it agrees with the version in \textup{(i)} almost
surely, also after translation. Put
\[
H_r(w)=\sigma_{\lambda_r^{\boldsymbol\iota,\tau}}
                  (Y_r^\alpha(0;w)),\qquad
M_t(W)=\int_0^tH_r(W)\,dW_r,
\]
where $M$ is a Borel path-valued representative of the It\^o integral.
Nonanticipation makes $H(W+h_\xi)$ adapted; both $H(W)$ and
$H(W+h_\xi)$ are bounded and continuous between the prescribed
switching times.

Choose refining deterministic partitions
$\pi_N=(t_j^{(N)})_{j=0}^{L_N}$ containing all switching times,
with $|\pi_N|\to0$, and define
\[
\begin{aligned}
H_r^{(N)}(w)
&=\sum_{j=0}^{L_N-1}H_{t_j^{(N)}}(w)
                        \one_{[t_j^{(N)},t_{j+1}^{(N)})}(r),\\
S_N(t,w)
&=\sum_{j=0}^{L_N-1}H_{t_j^{(N)}}(w)
   [w(t_{j+1}^{(N)}\wedge t)-w(t_j^{(N)}\wedge t)].
\end{aligned}
\]
The translated sum is exactly the sum of the It\^o integral of
$H^{(N)}(W+h_\xi)$ against $W$ and its Lebesgue integral against
$\dot h_\xi$. Dominated convergence gives
\[
r_N:=\E\int_0^T
 \|H_r^{(N)}(W+h_\xi)-H_r(W+h_\xi)\|_{\mathrm{HS}}^2\,dr\to0.
\]
The It\^o isometry and Doob's maximal inequality bound the expected
squared uniform stochastic-integral error by $4r_N$.
The uniform finite-variation error is bounded by
$\int_0^T\|H_r^{(N)}(W+h_\xi)-H_r(W+h_\xi)\|
|\dot h_\xi(r)|\,dr$, which converges almost surely to zero by
dominated convergence, since $\dot h_\xi\in L^1$.
The same stochastic-integral estimate gives
$S_N(\cdot,W)\to M(W)$ uniformly in probability, hence
$S_N(\cdot,W+h_\xi)\to M(W+h_\xi)$ in that sense by
Cameron--Martin equivalence. Uniqueness of limits in probability
therefore yields, almost surely as paths,
\[
M_t(W+h_\xi)
=\int_0^tH_r(W+h_\xi)\,dW_r
 +\int_0^tH_r(W+h_\xi)\dot h_\xi(r)\,dr.
\]
Translating the unshifted equation now shows that
$\widetilde Y_t=Y_t^\alpha(0;W+h_\xi)$ solves the It\^o equation
with initial point $x$, drift
$b_{\lambda_t^{\boldsymbol\iota,\tau}}(\widetilde Y_t)
+\sigma_{\lambda_t^{\boldsymbol\iota,\tau}}(\widetilde Y_t)a(t)\xi$,
and diffusion $\sigma_{\lambda_t^{\boldsymbol\iota,\tau}}
(\widetilde Y_t)$. These coefficients are globally Lipschitz in the
state, with bounded $a(t)\xi$, so pathwise uniqueness identifies
$\widetilde Y=Y^\alpha(\xi;W)$ almost surely as paths.
This proves \textup{(iv)} for the fixed pair $(\alpha,\xi)$.
\end{proof}

\subsection{Coupled Lyapunov matrices for the network}\label{app:network-lyapunov}
We use the connected graph and matrices $D_e,Q^0$ from
Section~\ref{sec:examples}.

\begin{lem}[Positive definite solutions of the coupled Lyapunov equations]\label{lem:network-lyapunov}
		There are unique symmetric matrices $P_e$ solving
		\begin{equation}\label{eq:network-lyapunov}
			D_e^\top P_e+P_eD_e+
			\sum_{f\ne e}q^0_{ef}(P_f-P_e)=-I.
		\end{equation}
		Every $P_e$ is positive definite.
	\end{lem}
	\begin{proof}
The matrices $D_e$ are symmetric and negative semidefinite, and
connectivity gives
\[
\bigcap_{e\in\mathcal E}\ker D_e
=\{v:v_1=0,\ v_j=v_k\text{ for every }\{j,k\}\in\mathcal E\}
=\{0\}.
\]
Thus $\|e^{tD_e}\|\le1$, with $|e^{tD_e}v|=|v|$ for $t>0$
only when $v\in\ker D_e$, in which case $e^{tD_e}v=v$.
For any sequence $e_0,\ldots,e_m$ visiting every edge and positive
$t_0,\ldots,t_m$, put $F=e^{t_mD_{e_m}}\cdots e^{t_0D_{e_0}}$.
Equality $|Fv|=|v|$ forces every factor to fix $v$, hence $v=0$.
Compactness of the unit sphere gives $\|F\|<1$.

Let $\Lambda^0$ have generator $Q^0$, and let
$\dot\Phi_t=D_{\Lambda^0_t}\Phi_t$, $\Phi_0=I_N$.
Write $\E_e$ for expectation with $\Lambda^0_0=e$, and fix $T>0$. From each initial edge $e$, irreducibility supplies
a finite admissible sequence visiting all edges; assign its visits
positive durations summing to $T$.
Continuity gives an open switching-time box on which
$\|\Phi_T\|\le r_e<1$.
Its event has probability $p_e>0$, since the selected transition
rates and exponential holding-time densities are positive.
Set $r=\max_e r_e<1$ and $p=\min_e p_e>0$.
Since $\|\Phi_T\|\le1$ on every path,
\[
\E_e|\Phi_Tv|^2
\le\theta|v|^2,\qquad \theta=1-p(1-r^2)\in(0,1).
\]

The Markov property iterates this estimate at multiples of $T$,
and pathwise nonexpansion controls the intervening times. Hence
\[
\E_e|\Phi_tv|^2
\le\theta^{\lfloor t/T\rfloor}|v|^2
\le\theta^{-1}e^{-\beta t}|v|^2,\qquad
\beta=-T^{-1}\log\theta>0.
\]
The reference jump linear system is therefore exponentially
mean-square stable. Apply the coupled Lyapunov criterion of
\citet[(3.4), Theorems~3.2--3.3, pp.~1594--1595]{FangLoparo2002}
with drift matrices $D_e$, rate matrix $Q^0$, and forcing matrices $I_N$ to obtain the unique symmetric solution
of \eqref{eq:network-lyapunov}, with every $P_e$ positive definite.
\end{proof}

\fund No funding was received for this research.
\competing The author declares no competing interests.
\section*{A note on the use of AI}
The author used GPT-6 Astra for language revision and to assist in
exploring proof strategies and checking calculations.

The suggestions were reviewed and revised during manuscript preparation.
The author takes full responsibility for the results, references,
and final text.


\begin{thebibliography}{99}
\footnotesize

\bibitem[Aida(1990)]{Aida1990FlowSupport}
\textsc{Aida, S.} (1990). Support theorem for diffusion processes on Hilbert spaces.
\emph{Publications of the Research Institute for Mathematical Sciences}
\textbf{26,} 947--965.

\bibitem[Bakhtin and Hurth(2012)]{BakhtinHurth2012}
\textsc{Bakhtin, Y. and Hurth, T.} (2012). Invariant densities for dynamical systems
with random switching. \emph{Nonlinearity} \textbf{25,} 2937--2952.

\bibitem[Basak et~al.(1999)]{BasakBisiGhosh1999}
\textsc{Basak, G.~K., Bisi, A. and Ghosh, M.~K.} (1999). Stability of degenerate
diffusions with state-dependent switching.
\emph{Journal of Mathematical Analysis and Applications} \textbf{240,} 219--248.

\bibitem[Bena\"im et~al.(2015)]{BenaimEtAl2015}
\textsc{Bena\"im, M., Le Borgne, S., Malrieu, F. and Zitt, P.-A.} (2015).
Qualitative properties of certain piecewise deterministic Markov processes.
\emph{Annales de l'Institut Henri Poincar\'e, Probabilit\'es et Statistiques}
\textbf{51,} 1040--1075.

\bibitem[Ben Arous and L\'eandre(1991)]{BenArousLeandre1991}
\textsc{Ben Arous, G. and L\'eandre, R.} (1991). D\'ecroissance exponentielle du noyau
de la chaleur sur la diagonale. II.
\emph{Probability Theory and Related Fields} \textbf{90,} 377--402.

\bibitem[Boyd et~al.(2006)]{BoydEtAl2006}
\textsc{Boyd, S., Ghosh, A., Prabhakar, B. and Shah, D.} (2006). Randomized gossip
algorithms. \emph{IEEE Transactions on Information Theory} \textbf{52,} 2508--2530.

\bibitem[Cloez and Hairer(2015)]{CloezHairer2015}
\textsc{Cloez, B. and Hairer, M.} (2015). Exponential ergodicity for Markov processes
with random switching. \emph{Bernoulli} \textbf{21,} 505--536.

\bibitem[Fang and Loparo(2002)]{FangLoparo2002}
\textsc{Fang, Y. and Loparo, K.~A.} (2002). Stabilization of continuous-time jump
linear systems. \emph{IEEE Transactions on Automatic Control}
\textbf{47,} 1590--1603.

\bibitem[Hairer(2009)]{Hairer2009SPDE}
\textsc{Hairer, M.} (2009). An introduction to stochastic PDEs.
arXiv:0907.4178v1.

\bibitem[Hairer and Mattingly(2011)]{HairerMattingly2011Harris}
\textsc{Hairer, M. and Mattingly, J.~C.} (2011). Yet another look at Harris'
ergodic theorem for Markov chains. In \emph{Seminar on Stochastic
Analysis, Random Fields and Applications VI}, Progress in Probability,
Vol.~63. Birkh\"auser, pp.~109--117.

\bibitem[Han et~al.(2014)]{HanEtAl2014}
\textsc{Han, Y., Lu, W., Li, Z. and Chen, T.} (2014). Pinning dynamic systems of
networks with Markovian switching couplings and controller-node set.
\emph{Systems \& Control Letters} \textbf{65,} 56--63.

\bibitem[Herzog and Mattingly(2015)]{HerzogMattingly2015}
\textsc{Herzog, D.~P. and Mattingly, J.~C.} (2015). A practical criterion for
positivity of transition densities. \emph{Nonlinearity} \textbf{28,} 2823--2845.

\bibitem[Hu et~al.(2019)]{HuNualartSunXie2019}
\textsc{Hu, Y., Nualart, D., Sun, X. and Xie, Y.} (2019). Smoothness of density for
stochastic differential equations with Markovian switching.
\emph{Discrete and Continuous Dynamical Systems, Series B} \textbf{24,} 3615--3631.

\bibitem[Jin et~al.(2018)]{JinLinPei2018}
\textsc{Jin, M., Lin, Y. and Pei, M.} (2018). Asymptotic behavior of a regime-switching
SIR epidemic model with degenerate diffusion.
\emph{Advances in Difference Equations} \textbf{2018,} 84.

\bibitem[Kunita(1986)]{Kunita1986Flows}
\textsc{Kunita, H.} (1986). \emph{Lectures on Stochastic Flows and Applications}.
Tata Institute of Fundamental Research Lectures on Mathematics and Physics,
Vol.~78. Springer-Verlag.

\bibitem[Kunwai(2022)]{Kunwai2022}
\textsc{Kunwai, K.} (2022). On Foster--Lyapunov criteria for exponential ergodicity
of regime-switching jump diffusion processes with countable regimes.
\emph{Journal of Applied Probability} \textbf{59,} 167--186.

\bibitem[Lazi\'c and Sandri\'c(2022)]{LazicSandric2022}
\textsc{Lazi\'c, P. and Sandri\'c, N.} (2022). On subgeometric ergodicity of
regime-switching diffusion processes.
\emph{Nonlinear Analysis: Hybrid Systems} \textbf{46,} 101262.

\bibitem[Li and Xi(2022)]{LiXi2022TV}
\textsc{Li, J. and Xi, F.} (2022). Exponential ergodicity for regime-switching
diffusion processes in total variation norm.
\emph{Discrete and Continuous Dynamical Systems, Series B} \textbf{27,} 6125--6146.

\bibitem[Liu et~al.(2013)]{LiuLinChen2013}
\textsc{Liu, X., Lin, H. and Chen, B.~M.} (2013). Structural controllability of
switched linear systems. \emph{Automatica} \textbf{49,} 3531--3537.

\bibitem[Mertz et~al.(2024)]{MertzNersesyanRissel2024}
\textsc{Mertz, L., Nersesyan, V. and Rissel, M.} (2024). Exponential mixing of
constrained random dynamical systems via controllability conditions.
\emph{SIAM Journal on Control and Optimization} \textbf{62,} 3266--3287.

\bibitem[Nersesyan and Raqu\'epas(2021)]{NersesyanRaquepas2021}
\textsc{Nersesyan, V. and Raqu\'epas, R.} (2021). Exponential mixing under
controllability conditions for SDEs driven by a degenerate Poisson noise.
\emph{Stochastic Processes and their Applications} \textbf{138,} 26--55.

\bibitem[Salamon(2016)]{Salamon2016Measure}
\textsc{Salamon, D.~A.} (2016). \emph{Measure and Integration}.
EMS Textbooks in Mathematics. European Mathematical Society.

\bibitem[Scagliotti(2023)]{Scagliotti2023}
\textsc{Scagliotti, A.} (2023). A gradient flow equation for optimal control
problems with end-point cost.
\emph{Journal of Dynamical and Control Systems} \textbf{29,} 521--568.

\bibitem[Shao(2015)]{Shao2015Wasserstein}
\textsc{Shao, J.} (2015). Ergodicity of regime-switching diffusions in Wasserstein
distances. \emph{Stochastic Processes and their Applications}
\textbf{125,} 739--758.

\bibitem[Shirikyan(2017)]{Shirikyan2017}
\textsc{Shirikyan, A.} (2017). Controllability implies mixing I. Convergence in the
total variation metric. \emph{Russian Mathematical Surveys} \textbf{72,} 939--953.

\bibitem[Sun(2007)]{Sun2007Constrained}
\textsc{Sun, Z.} (2007). Reachability analysis of constrained switched linear
systems. \emph{Automatica} \textbf{43,} 164--167.

\bibitem[Sun et~al.(2002)]{SunGeLee2002}
\textsc{Sun, Z., Ge, S.~S. and Lee, T.~H.} (2002). Controllability and reachability
criteria for switched linear systems. \emph{Automatica} \textbf{38,} 775--786.

\bibitem[Tong and Majda(2016)]{TongMajda2016}
\textsc{Tong, X.~T. and Majda, A.~J.} (2016). Moment bounds and geometric ergodicity
of diffusions with random switching and unbounded transition rates.
\emph{Research in the Mathematical Sciences} \textbf{3,} 41.

\bibitem[Xi and Zhu(2017)]{XiZhu2017}
\textsc{Xi, F. and Zhu, C.} (2017). On Feller and strong Feller properties and
exponential ergodicity of regime-switching jump diffusion processes
with countable regimes.
\emph{SIAM Journal on Control and Optimization} \textbf{55,} 1789--1818.

\bibitem[Xi et~al.(2021)]{XiZhuWu2021}
\textsc{Xi, F., Zhu, C. and Wu, F.} (2021). On strong Feller property, exponential
ergodicity and large deviations principle for stochastic damping
Hamiltonian systems with state-dependent switching.
\emph{Journal of Differential Equations} \textbf{286,} 856--891.
\end{thebibliography}
\end{document}